\documentclass[11pt, reqno]{amsart}    	
\usepackage[margin=1in]{geometry}
\usepackage{setspace}
\usepackage{graphicx}	
\usepackage{subcaption}
\usepackage{float}
\usepackage[toc]{appendix}
\usepackage[english,algoruled,vlined,resetcount,linesnumbered]{
			algorithm2e}

\usepackage[backref=page]{hyperref}
\hypersetup{%
	colorlinks,
	linkcolor={red!50!black},
	citecolor={green!50!black},
	urlcolor={blue!50!black}
}
\usepackage{amsmath}%
\usepackage{}
\usepackage{amsthm}
\usepackage{amsfonts}%
\usepackage{amssymb}%
\usepackage{graphicx}
\usepackage{mathrsfs}
\usepackage{lineno}
\usepackage{bm}
\usepackage[draft]{todonotes}   
\usepackage{hyperref}
\usepackage{lineno}
\usepackage{setspace}
\usepackage{enumerate}

\usepackage{paralist}	

\usepackage{amssymb}
\usepackage{amsmath, amsthm, amssymb}

\newtheorem{thm}{Theorem}[section]
\newtheorem{conj}[thm]{Conjecture}
\newtheorem{prop}[thm]{Proposition}
\newtheorem{claim}[thm]{Claim}
\newtheorem{lemma}[thm]{Lemma}

\newtheorem{fact}[thm]{Fact}
\newtheorem*{emc}{Erd\H{o}s Maching Conjecture}

\theoremstyle{definition}
\newtheorem{defi}[thm]{Definition}
\def\eps{\varepsilon}
\def\cT{\mathcal T}
\def\cY{\mathcal Y}
\def\ex{\text{ex}}

\usepackage{tikz}
\usepackage{tikz} \usetikzlibrary{shapes,arrows}

\newcommand*{\rom}[1]{\expandafter{\romannumeral #1\relax}}

\usepackage{amsmath}
\numberwithin{equation}{section}
\begin{document}
\title[Large $ Y_{3,2} $-tilings]{Large $ Y_{3,2} $-tilings in $ 3 $-uniform hypergraphs}
\author{Jie Han}
\address{School of Mathematics and Statistics and Centre for Applied Mathematics, Beijing Institute of Technology, Beijing, China}
\email{\tt han.jie@bit.edu.cn}

\author{Lin Sun}
\address{School of Mathematics, Shandong University,
	Jinan, China.}
\email{linsun@mail.sdu.edu.cn, ghwang@sdu.edu.cn}

\author{Guanghui Wang}

\begin{abstract}
	
Let $Y_{3,2}$ be the $3$-graph with two edges intersecting in two vertices.
We prove that every $3$-graph $ H $ on $ n $ vertices with  at least $ \max \left \{ \binom{4\alpha n}{3}, \binom{n}{3}-\binom{n-\alpha n}{3} \right \}+o(n^3) $ edges contains  a $Y_{3,2}$-tiling covering more than $ 4\alpha n$ vertices,  for sufficiently large $ n $ and $0<\alpha< 1/4$. 
The bound on the number of edges is asymptotically best possible and solves a  conjecture of the authors for $3$-graphs that generalizes the Matching Conjecture of Erd\H{o}s.
\end{abstract}	
\maketitle	
\section{Introduction}
Given $ k\ge2 $, a \emph{$ k $-uniform hypergraph} $H$ (in short, \emph{$ k $-graph}) consists of a vertex set $V$ and an edge set $E\subseteq \binom{V}{k}$, where $ \binom{V}{k} $ denotes the family of all $ k $-subsets of $ V $. 
We denote by  $ e(H):=|E| $ the number of edges in $H$.
A \emph{matching} in $ H $ is a set of disjoint edges of $H$.
Erd\H{o}s \cite{MR260599} conjectured the edge condition that guarantees a matching of a certain size in 1965.
\begin{emc}
Let $ n,s,k $ be three positive integers such that $ k\ge2 $ and $ n \ge k(s+1)-1 $. 
If $H$ is a $ k $-graph on $ n $ vertices which does not have a matching of size $s+1$, then
\[ e(H)\le \max \left \{ \binom{k(s+1)-1}{k}, \binom{n}{k}-\binom{n-s}{k} \right \}. \]
\end{emc}
For the importance of this conjecture, we quote from \cite{0THE}: ``It was one of the favorite problems of Erd\H{o}s."
The bounds in the conjecture come from two extremal constructions: an $ n $-vertex $k$-graph consisting of a complete $k$-graph on $k(s+1)-1$ vertices and $n-k(s+1)+1$ isolated vertices, and another construction is a complete $k$-graph on $n$ vertices with a complete $k$-graph on $n-s$ vertices removed.
For $ k\le3 $, the conjecture was settled in \cite{1959On,Frankl2017On,2012On,Tomasz2014On}.
For general $ k$, Erd\H os \cite{MR260599} himself proved the conjecture for $ n>n_{0}(k,s) $.
Subsequent improvements on $n_0$ have been obtained by various authors~\cite{BOLLOB1976SETS,2013Improved,0THE,2012The} and the current state of the art is $ n_0\le \frac{5}{3}sk-\frac{2}{3}s $ for sufficiently large $ s$ by Frankl and Kupavskii \cite{0THE}. 
Also, Kolupaev and Kupavskii~\cite{kol}  proved the conjecture for $ n\le(s+1)(k+\frac{1}{100k}) $, improving a result of Frankl \cite{Frankl2017Proof}.

In this paper we study the following generalization of the Matching Conjecture.
Given two $k$-graphs $F$ and $H$, an \emph{$ F $-tiling} (\emph{integer $ F $-tiling}) of $H$ is a family of vertex-disjoint copies of $F$  in $ H $. 
By the \emph{size} of an $ F $-tiling we mean the number of copies of $F$ contained in it. 
When $F$ is a single edge, an $F$-tiling is just a matching. 
For $k>b\ge 0$, let $Y_{k,b}$ be the $k$-graph consisting of two edges that intersect in exactly $b$ vertices. 
What is the largest number of edges in a $ k $-graph on $n$ vertices containing no  $Y_{k,b}$-tiling of size more than $s$?
We assume that $ n \ge (2k-b)(s+1)-1$ since otherwise the question is trivial. 
Here are two natural candidate extremal $ k $-graphs: a $k$-graph on $n$ vertices consisting of a complete $k$-graph on $(2k-b)(s+1)-1$ vertices and $n-(2k-b)(s+1)+1$ isolated vertices, and a  complete $k$-graph on $n$ vertices with a complete $k$-graph on $n-s$ vertices removed.
Moreover, for a $ k $-graph $ F $,  the Tur\'an number $ \ex( n , F ) $ is the maximum number of edges in a $ k $-graph on $ n $ vertices which does not contain $ F $ as a subgraph. Since $ \ex( n , Y_{k,b})  $ is $ o(n^k) $ (see Theorem~\ref{FrFu}), in each example, one can replace the (large) independent set by a $Y_{k,b}$-free $k$-graph.
Gan, Han, Sun and Wang \cite{large} considered this natural generalization of the matching problem to $Y_{k,b}$-tilings and proposed the following conjecture. 
A more general conjecture was proposed by Lang~\cite{lang2023tiling}.
\begin{conj}\cite{large} 
Let $ n,s,k,b $ be positive integers such that $ k>b>0$ and $ n \ge (2k-b)(s+1)-1$. 
If $H$ is a $ k $-graph on $ n $ vertices which does not have a $Y_{k,b}$-tiling of size $s+1$, then
\[ 
e(H)\le \max \left \{ \binom{(2k-b)(s+1)-1}{k}, \binom{n}{k}-\binom{n-s}{k} \right \}+o(n^k). 
\]
\end{conj}
The case $ s=0$ is an old conjecture of Erd\H{o}s \cite{MR0409246} and was resolved by Frankl and F\"uredi \cite{Peter1985Forbidding} (see Theorem~\ref{FrFu}). 
The case $ k=2$ and $ b=1 $ was confirmed by Grosu  and Hladk\'y \cite{MR2889516}.
Gan, Han, Sun and Wang \cite{large} confirmed the conjecture in some cases.
In this paper we verify the conjecture for $ k=3 $ and $ b=2 $, which extends the result of \cite{large}  in which the same result is proved for $ \alpha \in (0, 1/7) $. 
\begin{thm}[Main result]\label{thm1}
For every $ \alpha, \gamma\in (0,1/4) $ there exists $ n_{0} $ such that the following holds for each integer $n\ge n_0$. 
Let $H$ be a $3$-graph on $n$ vertices such that
\[ e(H)\ge \max \left \{ \binom{4\alpha n}{3}, \binom{n}{3}-\binom{n-\alpha n}{3} \right \}+\gamma n^3.\]
Then $H$ contains a $Y_{3,2}$-tiling covering more than $ 4\alpha n$ vertices.		
\end{thm}
Note that $Y_{3,2}$ can be viewed as a $3$-uniform loose cycle on four vertices and thus denoted by $C_4^3$, $\mathcal C_4^3$ or $C_2^3$ (a loose cycle of length two) by other authors. 
K{\"u}hn and  Osthus \cite{Daniela2006Loose},  Czygrinow, Debiasio  and  Nagle  \cite{2013Tiling} determined the minimum $ 2 $-degree threshold forcing perfect $Y_{3,2}$-tilings and the corresponding minimum vertex-degree threshold was studied in~\cite{2015Minimum} by Han and Zhao.
In addition, $Y_{k,b}$-tilings  have important applications on Dirac-type problems for Hamilton cycles in hypergraphs, see \cite{large,52,Jie2015Minimum,Han2015Minimum}.

Throughout the rest of the paper, we write $ Y:=Y_{3,2} $ for brevity.

\subsection*{Proof ideas.}
We first remark that we do not know how to use the powerful shifting technique in our context, which has been a crucial tool in studying large matchings. Indeed, the classical (edge) shifting may increase the size of a maximum $ Y $-tiling.

To prove Theorem~\ref{thm1}, we use the regularity method which reduces the $Y$-tiling problem to analyzing an auxiliary structure.
Indeed, as already used in~\cite{large}, we analyze a \emph{$\{Y, E\}$-tiling}, which is a $3$-graph consisting of vertex-disjoint union of a $Y$-tiling and a matching.
Our key new ingredient is Proposition~\ref{prop}, which characterizes the structure of a suboptimal $Y$-tiling and creates extra space for improving the tiling recursively.
Together with several other new ideas, this allows us to bridge between the two ``far-away'' extremal examples and solve the $Y$-tiling problem for the full range.
Moreover, the new proof is completely self-contained and does not depend on or use the proof in~\cite{large}.
At last, we use the fractional tilings and a trick in~\cite{Su} to avoid repeated use of the regularity lemma.

\subsection*{Notation}
Given a $ 3 $-graph $ H $  on $ V $, for $ U \subseteq V$, we write $ H-U $ for $ H[V\setminus U] $.
If $ U=\{v\} $ is a singleton, then we write $ H-v $ rather than $ H-\{v\} $.
We write $ F\subseteq H $ if $ F$ is a subgraph of $ H $.
Given a vertex $ v $ and vertex set $ S $ in the $3$-graph $ H $, we denote by $ \deg_{H}(v,S) $ the number of edges that contain $ v $ and two vertices from $ S $.

Given a (2-)graph $ G=(V,E) $, $A,B\subseteq V$, $ G[A,B] $ is the subgraph of $ G $ with $ E(G[A,B])=\{e\in E(G):|e\cap A|=|e\cap B|=1\} $.

Throughout the rest of this paper, we write $x\ll y\ll z$ to mean that we can choose constants from right to left, that is, for any $z>0$, there exist functions $f$ and $g$ such that, whenever $y\leq f(z)$ and $x\leq g(y)$, the subsequent statement holds. 
Statements with more variables are defined similarly.	

\section{Preliminaries}
In this section, we review the weak hypergraph regularity method, which is  a straightforward extension of Szemer\'{e}di's  regularity lemma for graphs \cite{szemeredi1975regular}. 
Let $ H = (V, E) $ be a $3$-graph and let $ A_1,A_2, A_3  $ be pairwise disjoint non-empty
subsets of $ V $. 
We define $ e(A_1,A_2, A_3 ) $ to be the number of crossing edges, namely, those
with one vertex in each $ A_i, i \in [3] $, and the density of $ H $ with respect to $ (A_1,A_2, A_3 ) $ as \[d(A_1,A_2, A_3 ) = \frac{e(A_1,A_2, A_3 )}{|A_1||A_2||A_3|}.\]
For $\delta > 0 $ and $ d \ge 0 $, a triple $ (V_1, V_2, V_3) $ of  pairwise disjoint subsets $ V_1, V_2, V_3\subseteq V $ is called  $ (\delta, d) $-regular,  if
\[|d(A_1,A_2, A_3)-d|\le \delta\]
for all triples of subsets $ A_i \subseteq V_i $, $  i \in [3] $ with $ |A_i| \ge \delta|V_i| $. 
We say $ (V_1, V_2, V_3) $ is
$ \delta $-regular if it is $ (\delta, d) $-regular for some $ d\ge 0 $. 
It is immediate from the definition that in a $ (\delta, d) $-regular triple $ (V_1, V_2, V_3)$, if $ V_i'\subseteq V_i$ has size $ |V_i'| \ge c|V_i| $ for some  $c>\delta$ , then
$ (V_1', V_2', V_3') $ is $ (\delta/c, d) $-regular.
\begin{lemma} \label{reg}(Weak regularity lemma). 
For all integers  $ t_0\ge1 $ and for every $ \delta> 0 $, there exists $T_0=T_0(t_0,\delta)$ such that for sufficiently large $ n $ and every $3$-graph $ H=(V,E) $ on $ n $ vertices, there exists a partition $ V=V_0\cup V_1\cup \dots \cup V_t$ satisfying 
	
(i) $ t_0\le t \le T_0$,
	
(ii) $|V_1|=\dots= |V_t|$ and $ |V_0|\le \delta n$, and
	
(iii) for all but at most $ \delta\binom{t}{3} $ triples $ \{i_1,i_2,i_3\}\subseteq[t]$,  $(V_{i_1},V_{i_2} ,V_{i_3})$ is $ \delta $-regular.
\end{lemma}
A vertex partition of a hypergraph $ H  $ as given by Lemma~\ref{reg}  will be referred to as a $ \delta $-regular $ t $-partition. 
For $ \delta> 0 $ and $ d> 0 $, we define the reduced hypergraph $R:=R(\delta,d,\mathcal{Q})$ of $ H $ with respect to such a $ \delta $-regular $ t $-partition $\mathcal{Q}$ on the vertex set $ [t]$ and $\{i_1,i_2,i_3\}\subseteq[t]$ is an edge if and only if  $(V_{i_1},V_{i_2} ,V_{i_3})$ is $ \delta $-regular and $d(V_{i_1},V_{i_2} ,V_{i_3})\ge d$.

The following lemma shows that	the reduced $3$-graph inherits the density of the original hypergraph. 
Its proof is standard, e.g., follows the lines of \cite[Proposition 16]{H2010Dirac} and thus omitted.

\begin{lemma} \label{Rdeg}
Suppose that  $1/n\ll \delta<d \ll \gamma,\mu,t_0$.	
Let $H$ be a $3$-graph on $n$ vertices with $e(H)\ge (\mu+ \gamma)  \binom{n}{3}$.
Then there exists $ T_0 $ and a $ \delta $-regular $ t $-partition $ \mathcal{Q} $ with $ t_0\le t \le T_0$ such that the reduced graph $R:=R(\delta,d,\mathcal{Q})$ satisfies $e(R)\ge (\mu+ \gamma/2)  \binom{|V(R)|}{3}$.	
\end{lemma}

\section{Large $ Y_{3,2} $-tilings }
In this section we deduce Theorem~\ref{thm1} from the following lemma and the regularity method.
Recall that given a $3$-graph $ H $, a $\{Y,E\}$-tiling of $ H $ is a family of vertex-disjoint copies of $Y$ and edges in $H$.
\begin{lemma}	\label{lem1}
Suppose $1/n \ll \varepsilon\ll\gamma$.
Let $H$ be a $3$-graph on $n$ vertices. 
Suppose $ \cT $ is a maximum $Y$-tiling in $ H $ and there exists no $\{Y,E\}$-tiling in $ H $ covering at least $ 4|\cT|+\varepsilon n $ vertices, then  
\[e(H)\le \max \left \{ \binom{ 4|\cT|}{3}, \binom{n}{3}-\binom{n- |\cT|}{3} \right \}+\gamma n^3. \]
\end{lemma} 
Next we observe that every regular triple in which the part sizes are not too different has an almost perfect  $Y$-tiling. 
\begin{prop} \label{p1}
Let $ 0<1/m\ll\delta\ll d $.
Suppose the triple $  (V_1,V_2,V_3) $ is $ (\delta,d) $-regular and $ m=|V_1|\le|V_2|\le|V_3|\le 3|V_1|-|V_2| $, then there exists a $Y$-tiling on $ (V_1,V_2,V_3) $ covering more than $ (1-4\delta)( |V_1|+|V_2|+|V_3|)$ vertices.
\end{prop}
\begin{proof} 
Let $ x_1=(3|V_1|-|V_2|-|V_3|)/4 $, $ x_2=(3|V_2|-|V_1|-|V_3|)/4 $ and  $ x_3=(3|V_3|-|V_1|-|V_2|)/4 $.
Then  $ 2x_1+x_2+x_3=|V_1| $,  $ x_1+2x_2+x_3=|V_2| $,  $ x_1+x_2+2x_3=|V_3| $ and $ 4(x_1+x_2+x_3)=|V_1|+|V_2|+|V_3| $.
Note that $0\le x_1\le x_2\le x_3$ and $ x_3=(3|V_3|-|V_1|-|V_2|)/4 \ge |V_1|/4=m/4 $.
We will construct a $Y$-tiling $M$  in $H[V_{1},V_{2},V_{3}]$ covering more than $ (1-4\delta)( |V_1|+|V_2|+|V_3|)$ vertices, by greedily finding copies of $Y$ using the regularity.	
	
Let $ H':=H[ V_1',V_2', V_3'] $, where $ V_i'\subseteq V_{i} $ with $ |V_i'|\ge \delta |V_i| $ for each $ i\in [3] $.
Since $(V_1,V_2,V_3) $ is $(\delta,d) $-regular, the triple $ (V_1', V_2', V_3') $ has density at least $ d-\delta $, that is, $ e(H')\ge (d-\delta) |V_1'||V_2'|| V_3'| $.
In particular, we can pick $ v_2\in V_i' $ and $ v_3\in V_j' $ for $ i\not=j, i,j\in[3] $ with $ \deg_{H'}(\{v_2,v_3\})\ge2 $, as otherwise $ e(H')\le2 |V_i'|| V_j'| $, which is a contradiction.
Then we choose $ v_1,v_1'\in V_p' $, where $ p\in [3]\setminus\{i,j\} $, such that $ \{v_1,v_2,v_3\}\in E(H')$ and $\{v_1',v_2,v_3\}\in E(H') $, giving a copy of $ Y $.
	
Thus it is possible to find a $Y$-tiling $M_1$ of size $\lfloor x_1\rfloor$ in $ H[ V_1,V_2, V_3] $, each of which intersects each of $V_{2}$ and $V_{3}$ in one vertex and $V_{1}$ in two vertices, and a $Y$-tiling $M_2$ of size $\lfloor x_2\rfloor$ in $ H[ V_1,V_2, V_3]\setminus M_1 $, each member of which intersects each of $V_{1}$ and $V_{3}$ in one vertex and $V_{2}$ in two vertices.
We repeatedly find disjoint copies of  $Y$ in $ H[ V_1,V_2, V_3]\setminus (M_1\cup M_2) $ each of which intersects each of $V_{1}$ and $V_{2}$ in one vertex and $V_{3}$ in two vertices, until fewer than $ 2\delta |V_3| $ vertices are left in $V_3$.
So when the procedure terminates, we obtain a $Y$-tiling denoted by $M_3$ in $ H[ V_1,V_2, V_3]\setminus (M_1\cup M_2) $, such that there are fewer than  $ \delta |V_3| $ vertices left in $V_1$ and $V_2$ respectively.
As a result, $ M_1\cup M_2 \cup M_3 $ is a $Y$-tiling on $ (V_1,V_2,V_3) $ covering more than $ 4\lfloor x_1\rfloor+4\lfloor x_2\rfloor+4\lfloor x_3\rfloor-4\delta|V_3|\ge 4(x_1+x_2+x_3)-12-4\delta|V_3|\ge (1-4\delta)( |V_1|+|V_2|+|V_3|)$ vertices.
\end{proof}

Fractional homomorphic tilings were first defined by Bu\ss, H\`an and Schacht \cite{Bu2013Minimum}.
In the proof of Theorem~\ref{thm1}, we view $\{Y,E\}$-tilings as (special) fractional hom($Y$)-tilings, defined as follows.
\begin{defi}	
Given a $3$-graph $H$, a fractional hom($Y$)-tiling in $H$ is a function $ h: V(H)\times E(H)\to [0,1] $ satisfying the following properties:
\begin{enumerate} \label{hom}
\item \label{hom1}if $ v\notin e $, then $ h(v,e)=0 $, 
\item \label{hom2} for every $ v\in V(H) $, $ h(v)=\sum_{e\in E(H)}h(v,e)\le1 $,
\item \label{hom3} for every $ e\in E(H) $, there exists a labeling of the vertices of $ e=uvw $ such that  \[h(u,e)\le h(v,e)\le h(w,e)\le3h(u,e)-h(v,e).\]
\end{enumerate}
We denote the smallest non-zero value of $ h $ by $ h_{\min} $ and 
the  weight of $ h $ is the sum over all values:
\[w(h):=\sum_{(v,e)\in V(H)\times E(H)}h(v,e).\]		
\end{defi} 
Note that  an  integer $Y$-tiling is a special fractional hom($Y$)-tiling. 
Given a  $Y$-tiling $\mathcal{F} $ in a $3$-graph $H$, we define a function $ h: V(H)\times E(H)\to  \{0,1/2, 1\} $ such that for any $ v\in V(H)$ and $ e\in E(H)  $, 
\[\label{frac}h(v, e) :=\begin{cases}
	1 & \text{if $ v\in e\subseteq Y\in \mathcal{F} $ and $ \deg_{Y}(v)=1 $,} \\
	1/2 & \text{if  $ v\in e\subseteq Y\in \mathcal{F} $ and $ \deg_{Y}(v)=2 $,}\\
	0& \text{otherwise.} \tag{3.1}
\end{cases}\]
It is a simple fact that $ h $ is a  fractional hom($Y$)-tiling in $H$. 

The following proposition allows us to convert a fractional hom($Y$)-tiling to an  integer $Y$-tiling. Its proof is an application of the regularity method.
\begin{prop} \label{ftoi}
Suppose  $ 1/n \ll \delta\ll d\ll\varepsilon,\alpha,b, 1/T_0$.	
Let $H$ be a $3$-graph on $n$ vertices and $\mathcal{Q} $ be a $ \delta $-regular $ t $-partition  of $ H $ with $  t \le T_0$.
Let $R:=R(\delta,d,\mathcal{Q})$ be the reduced hypergraph.  
Suppose  $R$ contains a fractional hom$(Y)$-tiling $ h $ with $ h_{\min}=b $ and $w(h)\ge 4(\alpha+\varepsilon) |V(R)| $. 
Then $H$ contains an integer $Y$-tiling of size more than $ \alpha n$.
\end{prop}
\begin{proof}	
Suppose that  $1/n\ll \delta\ll d\ll\varepsilon,\alpha,b, 1/T_0$. 
Let $H$ be a $3$-graph on $n$ vertices and $R:=R(\delta,d,\mathcal{Q})$ be the reduced hypergraph of $ H $ with respect to the partition  $ \mathcal{Q}:=\{V_0, V_1, \dots , V_t\} $, where $t\le T_0$. 
Let $ m:=(n-|V_0|)/t\ge(n-\delta n)/t $ be the size of each $ V_i, 1\le i\le t $.	
Suppose  $R$ contains a  fractional hom($Y$)-tiling $ h $ : $V(R)\times E(R)\to [0,1]$ such that $ w(h)\ge 4(\alpha+\varepsilon) t$. 

For each $ V_i $, $ i\in[t] $, by (\ref{hom2}) in Definition~\ref{hom}, $ V_i $ can be subdivided into a collection of pairwise disjoint sets $ (U_{i}^{e})_{V_i\in e\in E(R)} $ of size $|U_{i}^{e}|=h(V_i,e)m$.	
We suppress floors and ceilings, as they will be absorbed by the error terms.
For any edge $ e=V_{i_1}V_{i_2}V_{i_3}\in E(R) $, the triple $(V_{i_1},V_{i_2},V_{i_3})$ is $(\delta,d)$-regular.
By  $ h_{\min}=b $, we get that $( U_{i_1}^{e},U_{i_2}^{e},U_{i_3}^{e}) $ is $(\delta/b,d)$-regular.
Using property~(\ref{hom3}) in Definition~\ref{hom} and Proposition~\ref{p1} to $ (U_{i_1}^{e},U_{i_2}^{e},U_{i_3}^{e}) $, we  obtain a $Y$-tiling covering more than $ (1-4\delta) (h(V_{i_1},e)+h(V_{i_2},e)+h(V_{i_3},e))m $ vertices. 
Applying this to all edges of $ R $, we get a $Y$-tiling covering more than
\[
\sum_{e=V_iV_jV_k \in R } (1-4\delta) (h(V_{i},e)+h(V_{j},e)+h(V_{k},e))m\ge(1-4\delta)w(h)m\ge
(1-4\delta)4(\alpha+\varepsilon) tm > 4\alpha n
\]
vertices, where we used $  \delta\ll\varepsilon $, that is, the rounding errors can be absorbed by the error terms.
\end{proof}

For a $ k $-graph $ F $ on $ n $ vertices $ \{v_1,\dots v_n\} $ and $ b>0 $, we say $ F\{b\} $ is the  \emph{$ b $-blow-up} of  $ F $ if there exists an ordered partition $ (V_1,\dots,V_n) $ of $ V(F\{b\}) $ such that $ |V_1|=\dots=|V_n|=b$ and we have that $ u_1u_2\dots u_k\in E(F\{b\}) $ if and only if $ u_{i}\in V_{j_{i}} $, $ i\in[k] $ for some $v_{j_{1}}v_{j_{2}}\dots v_{j_{k}}\in E(F) $.
In this case, we say that each edge $ u_1u_2\dots u_k $ is a clone of $ v_{j_{1}}v_{j_{2}}\dots v_{j_{k}} $ and each vertex of $ V_i $ is a clone of~$ v_i $.

\begin{proof}[Proof of Theorem~\ref{thm1}]	
Choose constants such that	\[1/n\ll \delta\ll d\ll\varepsilon'  \ll \varepsilon \ll \gamma \ll \alpha <1 /4,1/T_0.\]
Let $H$ be a $3$-graph on $n$ vertices with 
\[\label{ec} e(H)\ge \max \left \{ \binom{4\alpha n}{3}, \binom{n}{3}-\binom{n-\alpha n}{3} \right \}+\gamma n^3.\]
Suppose $H$ contains no $Y$-tiling covering more than $ 4\alpha n$ vertices. 	
Apply  Lemma~\ref{Rdeg} to $H$, and let $V_{1}, \dots,V_{t}$ be the clusters of the partition $ \mathcal{Q} $ obtained with  $t\le T_0$. 
Let $R:=R(\delta,d,\mathcal{Q})$ be the reduced $3$-graph on these clusters.
If $ R $ contains  a fractional hom($Y$)-tiling $ h $ : $V(R)\times E(R)\to [0,1]$ such that $ w(h)\ge 4(\alpha+\varepsilon') t$, which can be achieved in Claim~\ref{fh}, then by Proposition~\ref{ftoi},  $H$ contains an integer $Y$-tiling covering more than $ 4\alpha n$ vertices, which is a contradiction.
So we may assume that $ R $ does not contain a fractional hom($Y$)-tiling of weight at least $ 4(\alpha+\varepsilon') t  $.
In particular, by the function \eqref{frac}, $ R $  contains no integer $Y$-tiling of size at least $ \alpha t+\varepsilon' t$.
The following claim uses Lemma~\ref{lem1} to build a large $ Y $-tiling in a large blow-up of $ R $. 
This trick was used in \cite{Su}.

\noindent\textbf{Claim.} There exists  $ 1\le j\le 1/\varepsilon  $ such that 	$ R\{4^j\} $  contains a $Y$-tiling of size at least $ (\alpha+\varepsilon' ) 4^{j}t $. 
\begin{proof}[Proof of claim]
Suppose $ \cT_0 $ is a maximum $Y$-tiling in $  R $ with $ |\cT_0|=\alpha_{0} t $. 
Then $ \alpha_{0}  <\alpha +\varepsilon'  $.
Let $ p:=\lfloor 1/\varepsilon\rfloor $ and $ \alpha_{i}:=\alpha_{0}+i\varepsilon $ for $ i\in[p].$
By Lemma~\ref{Rdeg},  	
\[
e( R)\ge\max \left \{ \binom{4\alpha t}{3}, \binom{t}{3}-\binom{t-\alpha t}{3} \right \}+\frac{\gamma t^3}{4}. 
\] 
We prove that there exists  $ 1\le j\le p $ such that $ R\{4^j\} $  contains a $Y$-tiling of size at least $ (\alpha+\varepsilon' ) 4^{j}t $ by the following procedure.
Indeed, suppose for $ i\in\{0,1,2,\dots,p\} $, 
$  R\{4^{i}\} $  contains a maximum  $Y$-tiling $ \cT_{i} $ of size at least $ \alpha_{i}4^{i} t $.
If  $ |\cT_{i}|\ge(\alpha +\varepsilon')4^{i}t $, then we are done by letting $ j=i $.
Otherwise,  $ |\cT_{i}|<(\alpha +\varepsilon')4^{i}t $.
Since for $ x\in \mathbb{N} $, $ \binom{xt}{3}-x^3\binom{t}{3}=O(t^{2}) $ and $ \varepsilon' \ll \gamma $, we have
\[
\begin{aligned}
e( R\{4^{i}\})=(4^{i})^{3}e(R)&\ge\max \left \{ \binom{4(\alpha+\varepsilon') 4^{i}t}{3}, \binom{4^{i}t}{3}-\binom{4^{i}t-(\alpha+\varepsilon')4^{i}t}{3} \right \}+\frac{\gamma (4^{i}t)^3}{5}\\&>\max \left \{ \binom{4|\cT_{i}|}{3}, \binom{4^{i}t}{3}-\binom{4^{i}t-|\cT_{i}|}{3} \right \}+\frac{\gamma (4^{i}t)^3}{5}.
\end{aligned}
\]
Using Lemma~\ref{lem1} with $ \gamma/5 $ in place of $ \gamma $ and $  4\varepsilon $ in place of $ \varepsilon $, we get a $\{Y,E\}$-tiling in $ R\{4^{i}\} $ covering at least $ 4\alpha_{i}4^{i} t+4\varepsilon4^{i} t=\alpha_{i+1} 4^{i+1}t $ vertices. 
Because the 4-blow-up of an edge contains a perfect $ Y $-tiling, we can construct in  $ R\{4^{i+1}\} $ a $Y$-tiling of size no less than $ \alpha_{i+1} 4^{i+1}t $.
The existence of such $ j $ is guaranteed by 
$ \alpha_{p} 4^{p}t=(\alpha_{0}+p\varepsilon) 4^{p}t\ge(\alpha+\varepsilon')4^{p}t $.
\end{proof}

Fix $ 1\le j\le 1/\varepsilon  $ as in the claim above such that $ R\{4^j\} $  contains a $Y$-tiling of size at least $ (\alpha+\varepsilon' ) 4^{j}t $. 
We observe the following claim:
\begin{claim} \label{fh}
$ R $ contains a fractional hom$ (Y) $-tiling $ h $ with $ h_{\min}\ge 1/4^j $ and $w(h)\ge 4(\alpha+\varepsilon')t $. 
\end{claim}
\begin{proof}
Suppose $ \mathcal{F} $	is a $Y$-tiling of size at least $ (\alpha+\varepsilon' ) 4^{j}t $ in $ R\{4^j\} $.
Let  $  \mathcal{F}' $ be the collection of edges $ e \in R\{4^j\} $ such that there exists $ e_1\in E(R\{4^j\}) $ together with $ e $ forming a member of $ \mathcal{F} $.
Note that $ |\mathcal{F}'|=2|\mathcal{F}|\ge2(\alpha+\varepsilon' ) 4^{j}t $.
Define functions $ f:E(R\{4^j\})\to E(R) $: for any $ e\in E(R\{4^j\}) $, let $ f(e):=L $ such that $ e $ is a clone of $ L $, and  $ g:V(R\{4^j\})\to V(R) $: for any $ v'\in V(R\{4^j\}) $, let $ g(v'):=v $ such that $ v' $ is a clone of $ v $.
Define another function $ w': V(R\{4^j\}) \times E(R\{4^j\}) \to \{0,1/2, 1\} $ such that for any $ v'\in V(R\{4^j\})$ and $ e\in E(R\{4^j\})  $, 
\[w'(v', e) :=\begin{cases}
	1 & \text{if $ v'\in e\subseteq Y_1\in \mathcal{F} $ and $ \deg_{Y_1}(v')=1 $,} \\
	1/2 & \text{if  $ v'\in e\subseteq Y_1\in \mathcal{F} $ and $ \deg_{Y_1}(v')=2 $,}\\
	0& \text{otherwise.}
\end{cases}\]

We will build a function $ w:\mathcal{F'}\times E(R)\times V(R\{4^j\})\to \left\{\frac{1}{4^j},\frac{1}{2\times4^j},0\right\} $. 
For any $ e\in \mathcal{F'}$, $ L\in E(R)  $ and $ v'\in V(R\{4^j\}) $, let 
\[w(e, L,v') :=\begin{cases}
	\frac{1}{4^j} & \text{if $  f(e)=L $ and $ w'(v', e)=1 $,} \\
	\frac{1}{2\times4^j} & \text{if  $ f(e)=L $ and $ w'(v', e)=1/2 $,}\\
	0& \text{otherwise.}
\end{cases}\]

Now we show that $ w $ gives rise to a  fractional hom($Y$)-tiling $ h $ with $  h_{\min}\ge 1/4^j  $ and $w(h)\ge 4(\alpha+\varepsilon')t $.
For any $ L\in E(R) $ and $ v\in V(R) $, we assign the weights 
\[ h(v,L):=\sum_{v':g(v')=v}\sum_{e\in \mathcal{F'}}w(e, L,v').\] 
Note that if $ v\notin L $, then $ h(v, L)=0 $. 
Indeed, suppose $ v\notin L $,  $ v'\in V(R\{4^j\}) $ with $ g(v')=v $ and $ e\in \mathcal{F'} $.
If $ f(e)\not=L $, then $ w(e, L,v')=0 $. 
Otherwise $ f(e)=L $, and we get that $ v'\notin e $.
So  $ w'(v',e)=0 $, implying $ w(e, L,v')=0 $. 
Thus $ w(e, L,v')=0 $ holds for any $ v'$ with $ g(v')=v $ and $ e\in \mathcal{F'} $, implying $ h(v,L)=0 $.
For every $ v\in V(R) $, we have  \[h(v):=\sum_{L\in E(R)}h(v,L)=\sum_{L\in E(R)}\sum_{v':g(v')=v}\sum_{e\in \mathcal{F'}}w(e, L,v')\le\frac{1}{4^j}\times 4^j=1 .\] 
Indeed, for any $ v' $ with $ g(v')=v $, if $ v' $ is not covered by $ \mathcal{F} $, then $  \sum_{L\in E(R)}\sum_{e\in \mathcal{F'}}w(e, L,v')=0 $. 
Otherwise suppose $ v'\in V(Y_1) $ with $ Y_1\in\mathcal{F} $ consisting of $ e_1, e_2 \in E(R\{4^j\}) $. 
Thus we have
\[ \sum_{L\in E(R)}\sum_{e\in \mathcal{F'}}w(e, L,v')=w(e_1, f(e_1),v')+w(e_2, f(e_2),v')=\frac{1}{4^j}.\]
So  $  h_{\min}\ge 1/4^j $.
Summing over all $ 4^j $ possible $ v' $ with $ g(v')=v $, we obtain $ h(v)\le 1 $.

For an arbitrary edge $ L\in E(R) $, by the definition of $ w $, there exists a labeling of the vertices of $ L=uvw $ such that $ ( h(u,L), h(v,L), h(w,L))=s_1\left(\frac{1}{4^j},\frac{1}{2\times4^j},\frac{1}{2\times4^j}\right)+s_2\left(\frac{1}{2\times4^j},\frac{1}{4^j},\frac{1}{2\times4^j}\right)+s_3\left(\frac{1}{2\times4^j},\frac{1}{2\times4^j},\frac{1}{4^j}\right) $, where $ 0\le s_1\le s_2\le s_3 \le4^j $ are integers.
Note that $ h(u,L)=\frac{2s_1+s_2+s_3}{2\times4^j} $, $ h(v,L)=\frac{s_1+2s_2+s_3}{2\times4^j} $,
$ h(w,L)=\frac{s_1+s_2+2s_3}{2\times4^j} $ and $ 3h(u,L)-h(v,L)- h(w,L)=\frac{2s_1}{4^j}$.
Thus $ h(u,L)\le h(v,L)\le h(w,L)\le3h(u,L)-h(v,L) $.

The total weight $ w(h) $ is 
\[
\begin{aligned}  
	w(h)&:=\sum_{(v,L)\in V(R)\times E(R)}h(v,L)\\
	&= \sum_{e\in \mathcal{F'}}\sum_{v\in V(R)}\sum_{v':g(v')=v}\sum_{L\in E(R)}w(e, L,v')\\
	&\ge 2(\alpha+\varepsilon' ) 4^{j}t\times \frac{2}{4^{j}}\ge 4(\alpha+\varepsilon' )t,
\end{aligned}
\]		
where we used that for any 	$ e\in \mathcal{F'} $, \[\sum_{v\in V(R)}\sum_{v':g(v')=v}\sum_{L\in E(R)}w(e, L,v')=\sum_{v'\in e}w(e, f(e),v')=\frac{1}{2\times4^j}+\frac{1}{2\times4^j}+\frac{1}{4^j}=\frac{2}{4^{j}} \]
and $ |\mathcal{F'}|\ge 2(\alpha+\eps')4^{j}t $.
\end{proof}
Claim~\ref{fh} contradicts our assumption that $ R $ does not contain a fractional hom($Y$)-tiling of weight at least $ 4(\alpha+\varepsilon') t  $. 
This contradiction concludes the proof.
\end{proof}

\section{Large $\{ Y,E\} $-tilings }
It remains to prove Lemma~\ref{lem1}, which is the goal of this section.
The following result gives an upper bound on the number of edges in $ k $-graphs with no copy of $ Y_{k, b} $.
\begin{thm}\cite{Erd1974Intersection,Peter1985Forbidding}  \label{FrFu}
	For $k > b \ge 0$, there exists an integer $n_k$ such that for any $n \ge n_k$, 	${\rm ex}(n,Y_{k, b})\le \binom{n-1}{k-1}$.
\end{thm} 

Here is an overview of the proof of Lemma~\ref{lem1}, which says that given a maximum $Y$-tiling in $H$, if one cannot find a significantly larger $\{Y, E\}$-tiling in $H$ (i.e., covering significantly more vertices), then we can give an upper bound on $e(H)$.
The strategy is to show that given a maximum $Y$-tiling $ \cT $ and the set of uncovered vertices $U$, there cannot exist many local configurations each of which allows us to find a $\{Y, E\}$-tiling on the union of $U$ and several (typically, three) copies of $Y$ in $\cT$ that covers \emph{more} vertices than $\cT$ does.
As mentioned in the introduction, a key ingredient in the proof is the algorithm in Proposition~\ref{prop}, which allows us to create extra flexibilities, that is, we can work with an extended version of $U$ (the set $W$ in Proposition~\ref{prop}).
We then reduce the proof to the inequality~\eqref{1.4}, which will be further reduced to Lemma~\ref{lm2}.
The rest of the proofs build on dealing with different kinds of forbidden structures.
The complexity comes from the fact that for each configuration, we have to exhibit an $ \Omega(n)$ number of vertex-disjoint copies of it.
For this we use other auxiliary structures, such as, digraphs and oriented $3$-graphs.


\begin{proof}[Proof of Lemma~\ref{lem1}]
Suppose we have the constants satisfying the following hierarchy
\[1/n\ll \varepsilon \ll \varepsilon'' \ll \varepsilon' \ll  \gamma.\]
Let $ \cT $ be  a maximum $Y$-tiling in $ H $ and suppose there exists no $\{Y,E\}$-tiling in $ H $ covering $ 4|\cT|+\varepsilon n $ vertices. 
We denote by $U$ the set of vertices not covered by $\cT$. 

\begin{prop} \label{prop}
There exist $ t_1, t_2\in \mathbb N $ with $ t_1+t_2=|\cT|$, a labelling of the members of $ \cT $ as $ {Y_1, Y_2, \dots, Y_{t_1+t_2}} $ and a set $ R $ of size $ t_1 $ with $ |R\cap V(Y_{i})|=1 $ for each $ i\in[t_1] $, such that the following holds. Let $W :=U\cup \left(\bigcup_{i\in[t_1]}V(Y_i)\setminus R\right)$.
\begin{enumerate}
\item  For each $ i\in\{t_1+1,t_1+2,\dots,t_1+t_2\}  $ and $ v\in V(Y_{i})$, we have $\deg(v,W)<\varepsilon'n^{2}  $.
\item 	For any $ W'\subseteq W $ with $ |W'|\le 10\varepsilon n $, there exists a $Y$-tiling in $ H[R\cup(W\setminus W')] $ of size $t_1$.
\end{enumerate}
\end{prop}

The item (2) in Proposition~\ref{prop} works in the following way.
Roughly speaking, when our arguments forbid a set of vertex-disjoint structures on the union of $U$ and the $Y$-tiling, by (2), one can extend such arguments to the union of $W$ and the vertices of the $Y$-tiling (e.g., it is easy to derive that $W$ is $Y$-free).

\begin{proof}
We prove the proposition by the following greedy algorithm.

\begin{procedure}[h]
	\caption{Construction of R()}
	\label{alg:main}
	\SetKwInOut{Input}{Data}
	\SetKwInOut{Output}{Output}
	\Input{an $n$-vertex $k$-graph $H$ with a  maximum $Y$-tiling $ \cT $ and the uncovered set $ U $.}
	Initialize $ W_0\leftarrow U $, $R\leftarrow\emptyset$, $ \cY\leftarrow\cT $ and $ i\leftarrow0 $.

	\While{there exists a vertex $ v $ in  a copy of $ Y \in\mathcal{Y}  $  such that $ \deg(v,W_{i})\ge\varepsilon'n^{2}  $  }
	{ Label this vertex as $v_{i+1}$ and label this copy as $Y_{i+1}$
	
	 $ W_{i+1}\leftarrow W_{i}\cup(V(Y_{i+1})\setminus \{v_{i+1}\}) $
	
$ \mathcal{Y}\leftarrow\mathcal{Y}\setminus \{Y_{i+1}\} $
	
	$ R\leftarrow R\cup \{v_{i+1}\} $
	 
	 $i \gets i + 1$}

\textbf{else} Stop.

\Output{vertex sets $ R $, $\mathcal Y$ and $ W_i $.} 
		
\end{procedure} 
From the algorithm, we suppose $ R=\{v_1,v_2,\dots, v_{t_{1}}\} $, where $ t_1+t_2=|\cT| $. Then $ \cT\setminus\mathcal{Y}=\{Y_{1},Y_{2},\dots,Y_{t_1}\} $. 
Label the copies of $ Y $ in $ \mathcal{Y} $ such that $ \mathcal{Y}=\{Y_{t_1+1},Y_{t_1+2},\dots,Y_{t_1+t_2}\} $.
Let $W :=W_{t_1}=U\cup (\bigcup_{i\in[t_1]}(V(Y_i)\setminus \{v_i\}))$.
Then for each $ i\in\{t_1+1,t_1+2,\dots,t_1+t_2\}  $, $ v\in V(Y_{i})$, we have $\deg(v,W)<\varepsilon'n^{2}  $. 
Note that $ \deg(v_i, W_{i-1})\ge\varepsilon'n^{2} $ for each $ v_i\in R  $.
So \[\label{1.1}\deg(v_j, W_{j-1})-\left(3\varepsilon'' n+3\times10\varepsilon n\sum_{k=1}^{1/\varepsilon''}3^{k}\right)|W|\ge\varepsilon'n^{2}-\varepsilon'n^{2}/2=\varepsilon'n^{2}/2, \tag{4.1}\] 
where the second term is the number of all  possible vertices we want to avoid during the following procedure.
Next we will use this degree condition to prove (2).

Fix  $ W'\subseteq W $ with $ |W'|\le 10\eps n $. 
Let $ W'\setminus U:=\{u_1,u_2,\dots, u_{q}\}  $, where $ q\le 10\varepsilon n $.
Note that $ W'\setminus U\subseteq  \bigcup_{i\in[t_1]}(V(Y_i)\setminus \{v_i\}$). For each $ v_i, i\in[t_1] $,  we will construct a copy of $ Y $ in $ H[R\cup(W\setminus W')] $ containing $ v_i$, such that all of them  together form a $ Y $-tiling of size $ t_1 $.
Without loss of generality, suppose $ u_i\in V(Y_{j_{i}}) $ for some $ j_{i}\in [t_1] $ and  $ j_1<j_2<\dots<j_q  $.
We will construct $ q $ vertex-disjoint directed rooted trees $ T_1,T_2, \dots ,T_q $ on $ R $ with roots $ v_{j_{1}},v_{j_{2}},\dots, v_{j_{q}} $ respectively, where  $ v_{j_{i}}\in V(Y_{j_{i}}) $ and  $ v_{j_{i}}\in R $, such that 
\begin{enumerate}
	\item if $  \overrightarrow{v_iv_j}\in E(T_{i'}) $ for some $ i'\in [q] $, then $ i-j\ge \eps'' n $,
	\item each vertex in the trees has at most three children.
\end{enumerate}
Hence the depth (the length of the longest path starting from the root in the tree)  of each tree is no more than $ \frac{t_1}{\varepsilon''n}\le1/\varepsilon''  $. 
Therefore the union of these trees contains at most $ 10\varepsilon n\sum_{k=0}^{1/\varepsilon''}3^{k} $ vertices.

We greedily construct the trees one by one as follows.
For $ i\le 0 $, let $ W_i:=U $.
We start with the empty tree and add the root vertex to it.
Throughout the process, as long as we add a new vertex $ v_i\in R $ to the tree (as a leaf), we choose a copy of $ Y $ denoted by $ Y(v_i) $ consisting of $ v_i $ and three other vertices from $ W_{i-\varepsilon'' n}$, vertex-disjoint from all existing vertices of the trees and the copies of $ Y $ associated with the tree vertices, which is possible by \eqref{1.1}. 
Specifically, suppose that $ S $ consists of all existing vertices of the trees.
Let $ W'_{i-\varepsilon'' n}:= W_{i-\varepsilon'' n}\setminus \bigcup_{v_{s}\in S }(V(Y_{s})\cup V(Y(v_{s}))) $.
We will choose $ Y(v_i) $ containing $ v_i $ in $ \{v_i\}\cup W'_{i-\varepsilon'' n} $.
Since $ |W_{i-1}|-|W_{i-\varepsilon'' n}|\le 3\varepsilon'' n $ and $ |\bigcup_{v_{s}\in S }(V(Y_{s})\cup V(Y(v_{s})))|\le 10\varepsilon n\sum_{k=1}^{1/\varepsilon''}3^{k}$, by \eqref{1.1}, we have $ \deg(v_j,  W'_{i-\varepsilon'' n})\ge\varepsilon'n^{2}/2>n/2 $, guaranteeing the existence of the desired $ Y(v_i) $.
Suppose $ V(Y(v_i)) = \{v_i, b_1, b_2, b_3\} $. 
If there is $ k\in [i-\eps'' n] $ such that $ b_j\in V(Y_k) $, then we add $ v_k $ as a child of $ v_i $. 
Thus, each $ v_i $ has at most three children in the tree.
The process terminates because the tree has depth at most $ 1/\eps'' $ and every vertex has at most three children. 
Other trees are constructed similarly, while being vertex-disjoint from all existing trees and the copies of $ Y $ associated with the tree vertices.

Let  $ \mathcal{Y}_{2}:=\{Y_{j}: j\in [t_1],  v_j\notin V(T_i)\ \text{for any} \ i\in[q]  \} $. 
By the construction of the trees, we obtain a desired copy of $ Y $ for each vertex of the trees such that all of them are pairwise disjoint.
Taking the union of $ \mathcal{Y}_{2} $ and all these newly built copies of $ Y $ associated with the tree vertices $\left\{Y(v_j): v_j\in \bigcup_{i\in[q]} V(T_i) \right\} $, we obtain a desired $ Y $-tiling of size~$ t_1 $.
\end{proof}
\begin{figure}[h]
	\begin{tikzpicture}
		[inner sep=2pt,
		vertex/.style={circle, draw=black!50, fill=black!50},
		]
		\draw[rounded corners] (-2.5,1.4) rectangle (1.4,0.6);
		\draw[rounded corners] (-6,0.4) rectangle (1.4,-1.8);

		\draw (2.8,0) ellipse  (0.45 and 1.6);
		\draw (4.1,0) ellipse  (0.45 and 1.6);
		\draw (6.8,0) ellipse  (0.45 and 1.6);
		
		\node at (-1.1,1) [vertex,color=red] {};
		\node at (0.7,1) [vertex,color=red] {};
		\node at (-1.9,1) [vertex,color=red] {};

	   	\node at (0.7,0) [vertex,color=black] {};
		\node at (0.7,-0.65) [vertex,color=black] {};
		\node at (0.7,-1.3) [vertex,color=black] {};
		\node at (-1.1,0) [vertex,color=black] {};
		\node at (-1.1,-0.65) [vertex,color=black] {};
		\node at (-1.1,-1.3) [vertex,color=black] {};
		\node at (-1.9,0) [vertex,color=black] {};
		\node at (-1.9,-0.65) [vertex,color=black] {};
		\node at (-1.9,-1.3) [vertex,color=black] {};
		\node at (2.8,1) [vertex,color=black] {};
		\node at (2.8,0.3) [vertex,color=black] {};
	
		\node at (2.8,-0.4) [vertex,color=black] {};
		\node at (2.8,-1.1) [vertex,color=black] {};
		\node at (4.1,1) [vertex,color=black] {};
		\node at (4.1,0.3) [vertex,color=black] {};
		\node at (4.1,-0.4) [vertex,color=black] {};
		\node at (4.1,-1.1) [vertex,color=black] {};
		\node at (6.8,1) [vertex,color=black] {};
		\node at (6.8,0.3) [vertex,color=black] {};
		\node at (6.8,-0.4) [vertex,color=black] {};
		\node at (6.8,-1.1) [vertex,color=black] {};
		
		\draw (-2.5,0.4) -- (-2.5,-1.8);
	\node at (-0.6,1.5) [label=above:$R$] {};
       \node at (-2.5,-2.5) [label=above:$W$] {};
       \node at (5,-2.5) [label=above:$\cY$] {};
	\node at (-4,-1) [label=above:$U$] {};
	   \node at (-0.2,-1) [label=above:$\dots$] {};
		\node at (-0.2,0.8) [label=above:$\dots$] {};
		\node at (5.5,-0) [label=above:$\dots$] {};
		\node at (2.9,1.6) [label=above:$Y_{t_1+1}$] {};
		\node at (4.2,1.6) [label=above:$Y_{t_1+2}$] {};
		\node at (6.9,1.6) [label=above:$Y_{t_1+t_2}$] {};
		
	\end{tikzpicture}
	\caption{The output of Proposition~\ref{prop}.}
	\label{figure3}
\end{figure}
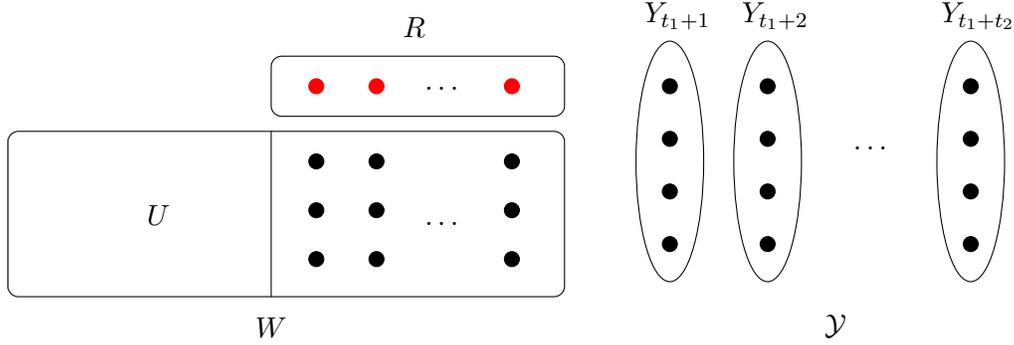
Fix the $ \cT,R,W$ returned by Proposition~\ref{prop}.
So $ t_1+t_2=|\cT|$, $ \cT=\{Y_{1},Y_{2},\dots,Y_{t_1+t_2}\}$ and $W :=U\cup \left(\bigcup_{i\in[t_1]}V(Y_i)\setminus R\right)$.
Suppose $ R=\{v_1,v_2,\dots, v_{t_{1}}\} $ with $ v_i \in V(Y_{i})$ for each $  i\in[t_1] $. 
Let $ \mathcal{Y}:=\{Y_{t_1+1},Y_{t_1+2},$ $\dots,Y_{t_1+t_2}\} $ (see Figure~\ref{figure3}).
We show the simple fact that $ W $ has no copy of $ Y $. 
\begin{claim} \label{c1}
$ W $ is $ Y $-free. 	
\end{claim}
\begin{proof}
Suppose $  W $ contains a copy of $ Y $ denoted by $ Y_0 $.
Applying Proposition~\ref{prop} (2) with $ W'=V(Y_0) $, we get that there exists a $Y$-tiling $\mathcal{Y}'  $ in $ H[R\cup(W\setminus V(Y_0))] $ of size $t_1$. 
Recall that $ \mathcal{Y}=\{Y_{t_1+1},Y_{t_1+2},\dots,Y_{t_1+t_2}\} $. 
Thus $\mathcal{Y}'\cup \mathcal{Y} \cup \{Y_0\}  $ is a $Y$-tiling of size $ t_1+t_2+1 $, contradicting the maximality of $ \cT $.
\end{proof}

Let $ y_1 $ be the number of edges that have nonempty intersections with $R$. 
Note that
\[ y_1\le \binom{n}{3}-\binom{n-t_1}{3}. \tag{4.2}\]  	
Let $ y_2:=e(H-R)=e(H)-y_1 $.
By Claim~\ref{c1} and Theorem~\ref{FrFu}, we have $e(H[W])\le n^2$. 
If $ t_2\le2 $, then the number of edges in $ H-R $ not contained in $ H[W] $ is at most $ 8n^{2} $.
So we have $ y_2\le9n^2 $ and	\[ e(H)= y_1+y_{2}\le\binom{n}{3}-\binom{n-t_1}{3}+9n^2\le\binom{n}{3}-\binom{n-|\cT|}{3}+\gamma n^3. \]
For the case $ t_2\ge 3 $, we shall prove the following in the next subsection:
\[\label{1.4} y_2\le\begin{cases}
	64\binom{t_2}{3}+\gamma n^{3} & \text{if $\frac{|W|}{t_2-2}\le\frac{9}{8}$,} \\
	37\binom{t_2}{3}+8\binom{t_2}{2}|W|+\gamma n^{3} & \text{if $\frac{|W|}{t_2-2}>\frac{9}{8}$.}
\end{cases} \tag{4.3}\]
Suppose \eqref{1.4} holds. 
If 	$\frac{|W|}{t_2-2}\le\frac{9}{8}$, then by $ t_1+4t_2+|W|=n $, we have  $8t_2\ge n-t_1$. Thus we have 
\begin{align*}  \label{e1}
	e(H)&= y_1+y_{2}\\
	&\le\binom{n}{3}-\binom{n-t_1}{3}+	64\binom{t_2}{3}+\gamma n^{3} \\&=\binom{t_1}{3}+\binom{t_1}{2}(n-t_1)+t_1\binom{n-t_1}{2}+64\binom{t_2}{3}+\gamma n^{3}\\&\le64\binom{t_1}{3}+8t_2\binom{t_1}{2}+t_1\binom{8t_2}{2}+64\binom{t_2}{3}+\gamma n^{3}\\&\le\binom{4t_1}{3}+4t_2\binom{4t_1}{2}+4t_1\binom{4t_2}{2}+\binom{4t_2}{3}+\gamma n^{3}\\&=\binom{4t_1+4t_2}{3}+\gamma n^{3}=\binom{ 4|\cT|}{3}+\gamma n^3, \tag{4.4}
\end{align*}
where we used  $4t_2\binom{4t_1}{2}+4t_1\binom{4t_2}{2}-8t_2\binom{t_1}{2}-t_1\binom{8t_2}{2}=4t_1t_2(7t_1-2)\ge0$ in the last inequality.
On the other hand, if $\frac{|W|}{t_2-2}>\frac{9}{8}$, then we first claim that $ 37\binom{t_2}{3}+8\binom{t_2}{2}|W|\le\binom{n-t_1}{3}-\binom{n-t_1-t_2}{3}  $. Indeed,
\[
\begin{aligned} 
	\binom{n-t_1}{3}-\binom{n-t_1-t_2}{3}&=\binom{t_2}{3}+\binom{t_2}{2}(n-t_1-t_2)+t_2\binom{n-t_1-t_2}{2}\\&=\binom{t_2}{3}+\binom{t_2}{2}(|W|+3t_2)+t_2\binom{|W|+3t_2}{2}
	\\&=\binom{t_2}{3}+\binom{t_2}{2}|W|+3\binom{t_2}{2}t_2+t_2\binom{|W|}{2}+t_2\binom{3t_2}{2}+3t_2^2|W|\\&\ge37\binom{t_2}{3}+8\binom{t_2}{2}|W|, 
\end{aligned}
\]
where we used  $t_2\le |W|$ in the last inequality. Thus

\begin{align*} \label{e2}
	e(H)= y_1+y_{2}&\le\binom{n}{3}-\binom{n-t_1}{3}+37\binom{t_2}{3}+8\binom{t_2}{2}|W|+\gamma n^{3} \\&\le\binom{n}{3}-\binom{n-t_1}{3}+\binom{n-t_1}{3}-\binom{n-t_1-t_2}{3}+\gamma n^{3}
	\\&=\binom{n}{3}-\binom{n-|\cT|}{3}+\gamma n^3. \tag{4.5}
\end{align*}
Combining  \eqref{e1} and \eqref{e2}, we get Lemma~\ref{lem1}.
\end{proof} 
Now it remains to prove \eqref{1.4} for $ t_2\ge3 $.

\subsection{More notation and tools}
Recall that $ \cT=\{Y_{1},Y_{2},\dots,Y_{t_1+t_2}\} $ is  a maximum $Y$-tiling in $ H $ and $\mathcal{Y}=\{Y_{t_1+1},\dots,Y_{t_1+t_2}\}$.
We need some more notation. 
Fix a pair $P=\{ Y_{p},Y_{q}\}\in \binom{ \cY }{2}$. 
Let $G'_{P}$ be the bipartite graph on $\left(V(Y_{p}),V(Y_{q})\right) $ with $\{u, v\}\in E(G'_P)$ if and only if there are at least $ 16\varepsilon n $ vertices $w$ in $W$, such that $ \{u,v,w\}\in E(H) $, where $u\in V(Y_{p})$, $v\in V(Y_{q})$.
Let $ \mathcal{D} $ be the collection of $ P\in \binom{\cY}{2} $ such that $ e(G'_P)\ge6 $  and $ G'_P $ has a vertex cover of two vertices which are in the same part.
Fix a triple $T=\{ Y_{i},Y_{j},Y_{k}\}\in \binom{ \cY }{3}$.
Let $ P_1=\{ Y_{i},Y_{j}\} $, $ P_2=\{ Y_{j},Y_{k}\} $ and $ P_3=\{ Y_{k},Y_{i}\} $.
Let $G'_{T}:=\bigcup_{i\in[3] }G'_{P_i}$ be the tripartite graph on $\left(V(Y_{i}),V(Y_{j}),V(Y_{k})\right) $. 
Let $ V(T):=V(Y_{i})\cup V(Y_{j}) \cup V(Y_{k}) $.
We have the following claim on the structure of $G'_{P}$. 
\begin{claim} \label{c3}
Suppose $ P\in\binom{\cY}{2} $, then $G'_{P}$ has no matching of size three and $e(G'_{P})\le8$.
\end{claim}
\begin{proof}
We claim that $G'_{P}$ has no matching of size three.
Indeed, suppose  $G'_{P} $ has a matching $\{e_1, e_2, e_3\}$.
By the definition of $G'_{P}$, we can find vertices $u_1,u_2,\dots, u_6\in W$ such that $e_1\cup \{u_1\}, e_1\cup \{u_2\}$, $e_2\cup \{u_3\}, e_2\cup \{u_4\}, e_3\cup \{u_5\}, e_3\cup \{u_6\}\in E(H)$, which gives a $Y$-tiling of size three denoted by $\mathcal{Y}'$.
Let $ W':=\{u_1,\dots, u_6\} $.
Applying Proposition~\ref{prop}, we get that there exists a $Y$-tiling $\mathcal{Y}''$ in $ H[R\cup(W\setminus W')] $ of size $t_1$. 
Thus $(\mathcal{Y}\setminus P)\cup\mathcal{Y}'\cup \mathcal{Y}''$ is a $Y$-tiling of size $ t_1+t_2+1 $, contradicting the maximality of $ \cT $.
So $G'_{P}$ has no matching of size three. 
By K\" onig's theorem \cite{2000Graph}, $G'_{P}$ has a vertex cover of size at most two, implying that $ e(G'_{P})\le 8 $.
\end{proof}

In order to obtain an  upper bound on the size of $ \mathcal{D} $. We need the following fact.
Let $ K_{2,3}^{+} $ be an oriented copy of (the complete bipartite graph) $ K_{2,3} $ with all six edges oriented from the part of size two to the part of size three (see Figure~\ref{figure1}). 
\begin{fact}\label{K}
Let $ 0<1/n\ll\varepsilon $.
Suppose $ D $ is a digraph on $ n $ vertices with $e(D)\ge \varepsilon n^{2} $.
Then $ D $ contains a copy of $ K_{2,3}^{+} $.
\end{fact}
\begin{proof}
We shall find a bipartite directed subgraph of $ D $ with at least half of the edges.
Let $ A\subseteq V(D) $ be a random subset given by $ \Pr(x\in A)=1/2 $ for each  $ x\in V(D) $ with these choices being mutually independent.
Let $ B:=V-A $ and  $ X $ be the number of edges in the bipartite digraph $ D[A,B] $.
Then $ \mathbb{E}[X]=e(D)/2 $.
We fix a choice of $ A $ such that $e(D[A,B])\ge e(D)/2\ge \eps n^2/2$. 
Without loss of generality, we assume that there are at least $\eps n^2/4$ edges which are directed from $ A$ to $B $. 
Using the K\H ov\'ari--S\'os--Tur\'an theorem, we get that  any $  K_{2,3} $-free bipartite graph with $ n $ vertices in each part has $ O(n^{3/2}) $ edges.
So we can find a copy of $ K_{2,3}^{+} $ in this bipartite digraph, and thus in $ D $.
\end{proof}
The following claim gives an upper bound on the size of $ \mathcal{D} $.
\begin{claim}\label{D}
$ |\mathcal{D}|\le \varepsilon'n^{2} $.
\end{claim}
\begin{proof}
Towards a contradiction, we suppose $ |\mathcal{D}|>\varepsilon'n^{2} $.
Let $ D $ be a digraph on $ \binom{\cY}{2} $ such that the ordered pair $ (Y_{i},Y_{j}) $ is an arc of  $ D $ if and only if the pair $ \{Y_{i},Y_{j}\}\in\mathcal{D} $ and $ G'_{\{Y_{i},Y_{j}\}} $ contains a vertex cover of size two  in $ V(Y_{j}) $.
So $ e(D)\ge \varepsilon'n^{2} $.
For each member of $ \cY $, we label the four vertices arbitrarily.
Fix an edge $ (Y_1,Y_2) $ of $ D $, and note that there are six possibilities for the location of the vertex cover of size two of $ G'_{\{Y_1,Y_2\}} $.
Thus, by averaging, there exists a subgraph $ D' $ of $ D $ such that $ e(D')\ge \eps' n^2/6 $ so that for all edges $ (Y',Y'') $  in $ D' $, the two vertices in the vertex cover of $ G'_{\{Y',Y''\}} $ are in the same location.
Using Fact~\ref{K} with $ \varepsilon'/6 $ in place of $ \varepsilon $, we get that $ D' $ contains a copy of $ K_{2,3}^{+} $ denoted by $ K $.

Suppose $ K $ has two parts  $ \{Y_{i_1},Y_{i_2}\} $ and $\{Y_{i_3},Y_{i_4},Y_{i_5}\} $ and contains all arcs $ (Y_{i_p},Y_{i_q}) $ for $ p\in[2], q\in\{3,4,5\} $.
Let $ I_{q} $, $ q\in\{3,4,5\} $ be the vertex cover of $ G'_{\{Y_{i_p},Y_{i_q}\}}  $ with $ |I_{q}|=2 $ and let $ G:=\bigcup_{ p\in[2], q\in\{3,4,5\}}G'_{\{Y_{i_p},Y_{i_q}\}} $.
Suppose $I:= \bigcup_{q\in\{3,4,5\}}I_{q}=\{v_1,v_2,v_3,v_4,v_5,v_6\} $ with $ d_{G}(v_6)\ge d_{G}(v_5)\ge d_{G}(v_4)\ge d_{G}(v_3)\ge d_{G}(v_2)\ge d_{G}(v_1) $.
Since $ e(G'_{\{Y_{i_p},Y_{i_q}\}})\ge6 $ for $p\in[2]$, $q\in\{3,4,5\} $, we have $ d_{G'_{\{Y_{i_p},Y_{i_q}\}}}(v_i)\ge6-4=2 $, where $ i\in[6] $.
So  $ d_{G}(v_i)\ge2+2=4 $ for each $ i\in[6] $.
Note that  $ d_{G}(v_6)\le 8 $, $ e(G)\ge6\times6=36 $ and $ I $ is a vertex cover of $ G $.
Then we have $ d_{G}(v_5)\ge 28/5 $, which implies that $ d_G(v_6)\ge d_{G}(v_5)\ge 6 $. Therefore, we can greedily find a matching $ \{e_1,e_2,\dots, e_6\} $ which consists of $ v_1,v_2,\dots, v_6 $ and six vertices in $ V(Y_{i_{1}}\cup Y_{i_{2}}) $, see Figure~\ref{figure1}.
Since each edge of $ \{e_i: i\in [6]\} $ can be extended to a copy of $Y$ by using two vertices in $W$, we may get six vertex-disjoint copies of $Y$, denoted by $ \mathcal{Y}' $.
Let $ W'\subseteq W $ be the set of 12 vertices used by the members of  $ \mathcal{Y}' $.
Applying Proposition~\ref{prop}, we obtain a $Y$-tiling $\mathcal{Y}''  $ in $ H[R\cup(W\setminus W')] $ of size $t_1$. 
Thus $\mathcal{Y}\setminus (\bigcup_{j\in[5]}Y_{i_j})\cup\mathcal{Y}'\cup \mathcal{Y}''  $ is a $Y$-tiling of size $ t_1+t_2+1 $, contradicting the maximality of $ \cT $.
\end{proof}
\begin{figure}
	\begin{tikzpicture}
		[inner sep=2pt,
		vertex/.style={circle, draw=black!50, fill=black!50},
		]
		\tikzstyle{square}=[rectangle, minimum size=2.3mm, fill=black, draw]
		\node at (-8,0) [vertex,color=black] {};		
		\node at (-4.5,-1) [vertex,color=black] {};		
		\node at (-8,2) [vertex,color=black] {};		
		\node at (-4.5,1) [vertex,color=black] {};		
		\node at (-4.5,3) [vertex,color=black] {};		
		\node at (-8,-0.8) [label=above:$Y_{i_{2}}$] {};
		\node at (-4.5,-1.8) [label=above:$Y_{i_{5}}$] {};
		\node at (-8,1.2) [label=above:$Y_{i_{1}}$] {};
		\node at (-4.5,0.2) [label=above:$Y_{i_{4}}$] {};
		\node at (-4.5,2.2) [label=above:$Y_{i_{3}}$] {};
		\draw[->,line width=1.1pt] (-8,-0) -- (-4.6,0.95);
		\draw[->,line width=1.1pt] (-8,-0) -- (-4.6,-1.05);
		\draw[->,line width=1.1pt] (-8,-0) -- (-4.6,2.95);
		\draw[->,line width=1.1pt] (-8,2) -- (-4.6,3.05);
		\draw[->,line width=1.1pt] (-8,2) -- (-4.6,1.05);
		\draw[->,line width=1.1pt] (-8,2) -- (-4.6,-0.95);
		
		\draw[rounded corners] (1.8,4.65) rectangle (3,2.3);
		\draw[rounded corners] (1.8,2) rectangle (3,-0.5);
		\draw[rounded corners] (1.8,-0.8) rectangle (3,-3.1);
		\draw[rounded corners] (-1.15,2) rectangle (0.05,-0.5);
		\draw[rounded corners] (4.75,2) rectangle (5.95,-0.5);
		
		\node at (-0.5,1.7) [vertex,color=black] {};
		\node at (-0.5,1.05) [vertex,color=black] {};
		\node at (-0.5,0.45) [vertex,color=black] {};
		\node at (-0.5,-0.2) [vertex,color=black] {};
		
		\node at (5.3,1.7) [vertex,color=black] {};
		\node at (5.3,1.05) [vertex,color=black] {};
		\node at (5.3,0.45) [vertex,color=black] {};
		\node at (5.3,-0.2) [vertex,color=black] {};
		\node at (2.4,1.7) [vertex,color=black] {};
		\node at (2.4,1.05) [square,color=blue] {};
		\node at (2.4,0.45) [square,color=blue] {};
		\node at (2.4,-0.2) [vertex,color=black] {};
		\node at (2.4,4.4) [vertex,color=black] {};
		\node at (2.4,3.8) [square,color=blue] {};
		\node at (2.4,3.2) [square,color=blue] {};
		\node at (2.4,2.6) [vertex,color=black] {};
		\node at (2.4,-1.1) [vertex,color=black] {};
		\node at (2.4,-1.7) [square,color=blue] {};
		\node at (2.4,-2.25) [square,color=blue] {};
		\node at (2.4,-2.8) [vertex,color=black] {};
		
		\draw (-0.5,1.7) -- (2.4,3.8);
		\draw (5.3,1.7) -- (2.4,3.2);
		\draw (-0.5,1.05) -- (2.4,1.05);
		\draw (5.3,0.45) -- (2.4,0.45);
		\draw (-0.5,-0.2) -- (2.4,-1.7);
		\draw (5.3,-0.2) -- (2.4,-2.25);

		
		\node at (-0.5,-1.2) [label=above:$Y_{i_{1}}$] {};
		\node at (3.4,2.1) [label=above:$Y_{i_{3}}$] {};
		\node at (3.4,-3.2) [label=above:$Y_{i_{5}}$] {};
		\node at (3.4,-0.8) [label=above:$Y_{i_{4}}$] {};
		\node at (5.4,-1.2) [label=above:$Y_{i_{2}}$] {};				
	\end{tikzpicture}
	\caption{A copy of $ K_{2,3}^{+} $ in $ D' $ and its underlying structure. The square vertices are $ v_1,v_2,\dots, v_6 $.}
	\label{figure1}
\end{figure}
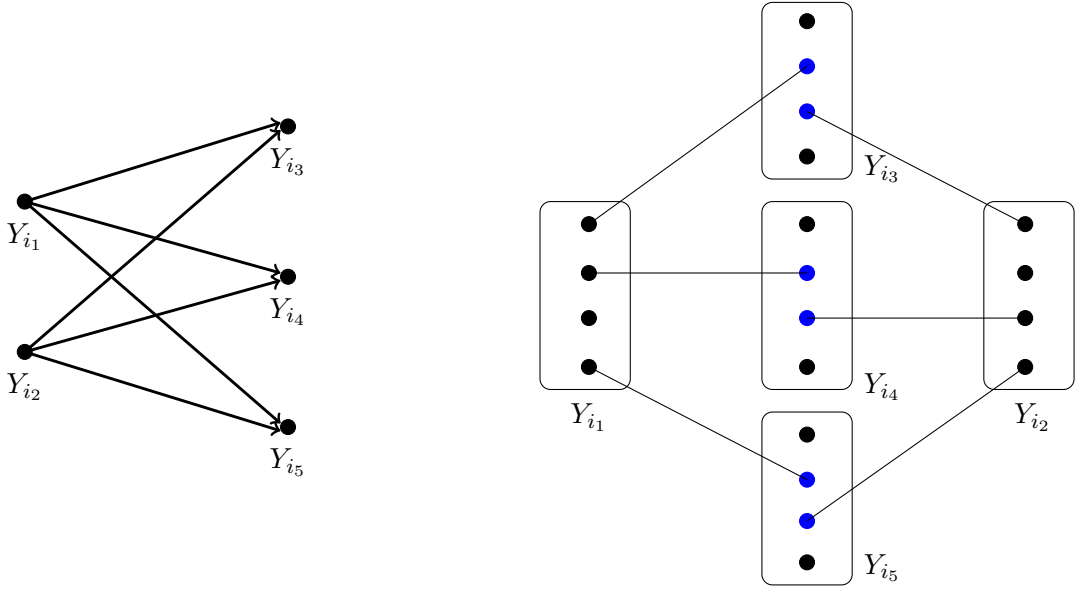

Recall that 
$ y_2 $ is the number of edges in $H-R=H[W\cup(\bigcup_{Y_{i}\in \mathcal{Y}}V(Y_{i}))] $.
The analysis of  $ y_2 $  is much more involved.	
For that we introduce the following notation.
Let $ YYY $ represent the set of the edges $\{v_1,v_2,v_3\}$ in $ H $, where $v_1\in V(Y_i)$, $v_2\in V(Y_j)$, $v_3\in V(Y_k)$ and $ i,j ,k \in \{t_1+1,\dots,t_1+t_2\}$ are all distinct. 
Let $ YYW $ represent the set of the edges $\{v_1,v_2,w\}$ in $ H $, where $  v_1\in V(Y_i)$, $v_2\in V(Y_j)$, $w\in W$ and $ i\neq j $.
We say such edge $\{v_1,v_2,w\}$ is reflected in $ G'_{\{Y_i,Y_j\}} $ if $ \{v_1,v_2\}\in E\left(G'_{\{Y_i,Y_j\}}\right) $.
Moreover, let $ YWW $ represent the set of the edges  $\{v_1,w_1,w_2\}$ in $ H $, where $  v_1\in V(Y_i)$, $w_1,w_2\in W $. 
Let $ H'' $ be a subgraph of $ H-R $, consisting of

\begin{enumerate}
	\item all the  edges in $YYW$ not reflected in any $G'_{P}$, $P \in \binom{ \cY }{2} $, 
	\item  all the  edges in $YYW$ containing some edge in  $ G'_{P} $ for some $ P\in \mathcal{D} $, 
	\item all the  edges in $YWW$, 
	\item all the edges in $ W$, and
	\item all the edges that have  at least two vertices covered by a single member of $\mathcal{Y}$.
\end{enumerate}
We claim that $ e(H'')\le 3\varepsilon'n^{3} $.
Indeed, by the definition of $G'_{P}$ and Claim~\ref{D}, we get that the number of edges in (1) and (2) is at most $ 16\varepsilon n\binom{n}{2}+\varepsilon'n^{3} $.
Since  $ \deg(v,W)<\varepsilon'n^{2}  $ for all $ v\in  V(Y_{i})$ with  $ Y_{i}\in\mathcal{Y} $, we have $ |YWW|\leq\varepsilon'n^{3} $. 
By Claim~\ref{c1} and Theorem~\ref{FrFu}, we see $e(H[W])\le n^2$. 
Note that the number of edges that have at least two vertices covered by a single $Y_{j}$ of $\mathcal{Y}$ is at most $ \binom{4}{2}|\cY|n\le6n^{2} $. 
So \[e(H'')\le16\varepsilon n\binom{n}{2}+\varepsilon'n^{3}+\varepsilon'n^{3}+n^2+6n^2\le3\varepsilon'n^{3}.\]
Let $ H':=(V(H), E(H)\setminus E(H'')) $.
Note that  $ E\left(H\left[W\cup \left(\bigcup_{Y_{i}\in \mathcal{Y}}V(Y_{i})\right)\right]\right)=E(H'')\cup YYY\cup YYW $. 
Thus
\[y_{2}=|YYY|+|YYW\cap E(H')|+e(H'')\le |YYY|+|YYW\cap E(H')|+3\varepsilon'n^{3}.\tag{4.6}\label{6.6}\]

For any pair $P\in \binom{ \cY }{2}$, let $ G_{P} $ be  a graph on $ V(G'_P) $ such that if $ P\in \mathcal{D} $, then $ E(G_{P})=\emptyset $, otherwise $ E(G_{P})=E(G'_{P}) $.
In other words, an edge $ e\in E(G_P) $ if and only if $ e $ is contained in at least $ 16\eps n $ edges of the form $ YYW $ in $ H' $.
For a triple $T=\{ Y_{i},Y_{j},Y_{k}\}\in \binom{ \cY }{3}$,
let $G_{T}:= G_{\{Y_{i},Y_{j}\}}\cup G_{\{Y_{j},Y_{k}\}} \cup G_{\{ Y_{k},Y_{i}\}} $ be the tripartite graph on $(V(Y_{i}),V(Y_{j}),V(Y_{k})) $, 	
and let $Q_{T}$ be the 3-partite 3-graph that is the  induced tripartite subgraph of $H'$  on $(V(Y_{i}),V(Y_{j}),V(Y_{k})) $. 
Suppose $f(T):= e(G_T) $ and $ g(T):=e(Q_{T}) $, for $ T\in \binom{ \cY }{3} $. 
By Claim~\ref{c3}, we get that  $ f(T)\le 3\times 8=24 $ for any $ T\in \binom{ \cY }{3} $.
Here is our estimate on $ f(T) $ and $ g(T) $, which is crucial in the proof of~\eqref{1.4}.
\begin{lemma} \label{lm2} 
For all but at most $\varepsilon' n^{3}$ triples $ T\in \binom{ \cY }{3} $, we have $ g(T)\le 64-\dfrac{9}{8}f(T) $. 
\end{lemma}
Now we prove \eqref{1.4} by assuming that Lemma~\ref{lm2} holds.
Since  $ t_2\ge3 $, we have $  |YYY|=|YYY\cap E(H')| = \sum_{T\in\binom{ \cY }{3}}g(T) $ and $ |YYW\cap E(H')|\le \frac{|W|}{t_2-2}\sum_{T\in\binom{ \cY }{3}}f(T) $.
So	\[y_{2}\le\sum_{T\in \binom{ \cY }{3} }g(T)+\frac{|W|}{t_2-2}\sum_{T\in \binom{ \cY }{3}}f(T)+3\varepsilon'n^{3}.\]
Let $ \mathcal{T}'\subseteq \binom{ \cY }{3} $ consisting of $ T\in \binom{ \cY }{3} $ such that for any $ T\in \mathcal{T}' $, $ g(T)\le 64-\dfrac{9}{8}f(T) $.
Note that $ s:=\left|\binom{ \cY }{3}\setminus\mathcal{T}'\right|\le \min\{\eps'n^3,(t_{2})^{3}\}\le(\eps'n^3)^{\frac{2}{3}}\cdot(t_{2})^{3\times\frac{1}{3}}=(\eps')^{\frac{2}{3}}n^2t_2 $.
Then  
\[
\begin{aligned} 
	y'&:=\sum_{T\in\binom{ \cY }{3}\setminus\mathcal{T}'}g(T)+\frac{|W|}{t_2-2}\sum_{T\in\binom{ \cY }{3}\setminus\mathcal{T}'}f(T) 
	\\&\le 64s+\frac{|W|}{t_2-2}\times48s 
	\\&\le64(\eps')^{\frac{2}{3}}n^2t_2+\frac{|W|}{t_2-2}\times48(\eps')^{\frac{2}{3}}n^2t_2
	\\&\le64(\eps')^{\frac{2}{3}}n^2t_2+48(\eps')^{\frac{2}{3}}n^3\times3
	\le208(\eps')^{\frac{2}{3}}n^{3}.
\end{aligned}
\]
For the case $\frac{|W|}{t_2-2}\le\frac{9}{8}$, we have 
\[
\begin{aligned} 
	y_{2}&\le\sum_{T\in \binom{\cY}{3}} \left( g(T) + \frac{|W|}{t_2-2} f(T) \right) +3\varepsilon'n^{3}
	\\&\le64\binom{t_2}{3}+\sum_{T\in \mathcal{T}'}\left(g(T)+\dfrac{9}{8}f(T)-64\right)+y'+3\varepsilon'n^{3}
	\le64\binom{t_2}{3}+\gamma n^{3}. 
\end{aligned}
\]
On the other hand, if	$\frac{|W|}{t_2-2}>\frac{9}{8}$, then
\[
\begin{aligned} 
	y_{2}&\le \sum_{T\in \binom{\cY}{3}} \left( g(T) + \frac{|W|}{t_2-2} f(T) \right) +3\varepsilon'n^{3}\\
	&=24\binom{t_2}{3}\frac{|W|}{t_2-2}+\sum_{T\in\mathcal{T}'}g(T)+\frac{|W|}{t_2-2}\sum_{T\in \mathcal{T}'}f(T)-24\binom{t_2}{3}\frac{|W|}{t_2-2}+y'+3\varepsilon'n^{3}\\&
	\le8\binom{t_2}{2}|W|+\sum_{T\in \mathcal{T}'}\left(g(T)-\frac{|W|}{t_2-2}(24-f(T))\right)+\gamma n^{3}\\&
	\le8\binom{t_2}{2}|W|+\sum_{T\in \mathcal{T}'}\left(g(T)-\dfrac{9}{8}(24-f(T))\right)+\gamma n^{3}
	\\&\le8\binom{t_2}{2}|W|+37|\mathcal{T}'|+\gamma n^{3}
	\le37\binom{t_2}{3}+8\binom{t_2}{2}|W|+\gamma n^{3}, 
\end{aligned}
\]
where we used $ f(T)\le24 $ for all $ T\in \binom{ \cY }{3} $ in the third inequality.
Thus the proof of \eqref{1.4} has been completed.
Now it remains to prove Lemma~\ref{lm2}. 	
The following facts will be useful in our proof.

\begin{fact} 	\cite[Fact 4]{large}\label{f0}
Let $a,b$ be integers with $b\ge a\ge2$. 
Let $H$ be a $3$-partite $3$-graph on $V_1,V_2,V_3$ with $|V_1|=|V_2|=a$, $|V_3|=b$ and no matching of size $a$. 
Then $e(H)\le(a-1)ab$.
\end{fact}	

For $ t<n $ and $ k\ge 2 $, let $ H^{(k)}(t,n) $ be the $k$-partite $k$-graph on $ kn$ vertices consisting of the complete $ k $-partite $ k $-graph with one part of size $ t $ and all other parts of size $ n $, and $ n-t $ isolated vertices.
Moreover, let $ H^{(k)}(t,n)^{-} $  be the (unique) $ k $-partite $ k $-graph obtained from $ H^{(k)}(t,n) $ with one edge removed.
The following result was proven in \cite{Markstr2011Perfect} and we also need a slight strengthening of it (Fact~\ref{f1}).
\begin{fact} \cite[Theorem 4]{Markstr2011Perfect} \label{f11}
For all integers $  k\ge 1 $, $ n \ge3 $, and $ 1 \le t\le n - 1 $, the maximum number of edges in a $k$-partite $k$-graph $ H $ with $n$ vertices in each class and no matching of size $t+1$ is $tn^{k-1}$ and $ e(H)=tn^{k-1}  $ if and only if $H\cong H^{(k)}(t,n) $.
\end{fact}

\begin{fact} \label{f1}
For all integers $ k\ge2$, $ n\ge4$ and $ 3\le t\le n-1 $,  $ H^{(k)}(t,n)^{-} $ is the only $k$-partite $ k $-graph with $ n $ vertices in each class and exactly $ tn^{k-1}-1 $ edges which contains no matching of size $ t+1 $.
\end{fact}
We postpone the proof of Fact~\ref{f1} to Section 4.3.

An oriented $3$-graph $ F $ is a $3$-graph where each edge is an ordered triple of vertices.
Let $ \Lambda$ be an  oriented $3$-graph consisting of two edges that intersect exactly in the third vertex. 
The following result allows us to find many vertex-disjoint copies of $ \Lambda$ in a dense oriented $3$-graph. 
\begin{lemma} \label{f4}
Let $ 0<1/n\ll\varepsilon $.
Suppose $ F $ is an oriented $3$-graph on $ n $ vertices with $e(F)> 12\varepsilon n^{3} $.
Then $ F $ contains $\varepsilon n$  vertex-disjoint copies of $ \Lambda$.
\end{lemma}
\begin{proof}
Suppose $ V(F):=\{v_1,v_2,\dots,v_n\} $.
We apply the following procedure iteratively and  get a subgraph of $ F $, denoted by $ F' $,  which satisfies that for each $v_{i}\in V $, $ |\{(j,k):(v_{j},v_{k},v_{i} )\in E(F')\}|$ is either zero or more than $ 6\varepsilon n^2 $. 
Whenever there is $ v_{i} $ that is in at most $  6\varepsilon n^2 $ edges  containing $ v_{i} $ as the third vertex, we delete all such edges.  
Note that when the process ends, the number of edges deleted is at most $ 6\varepsilon n^3 $.

We can greedily find $\varepsilon n$ vertex-disjoint copies of $ \Lambda$ in $ F' $ one by one. 
Indeed, suppose we have found $s<\varepsilon n$  vertex-disjoint copies of $ \Lambda$.
After the removal of their vertices from $ F' $, we obtain $ F'' $ with $ e(F'')\ge 12\varepsilon n^{3}-6\varepsilon n^3-5\varepsilon n^{3}\ge \varepsilon n^{3} $.
For  each $v_i\in V(F'') $, $ |\{(j,k):(v_j,v_k,v_i )\in E(F'')\}|$ is either zero or more than $ 6\varepsilon n^2-5sn> 6\varepsilon n^2-5\varepsilon n^2=\varepsilon n^2 $. 
So we can fix an edge $ (v_{i_1},v_{i_2},v_{i_3})\in E(F'') $, then  $ |\{(j,k):(v_j,v_k,v_{i_3})\in E(F'')\}|\ge\varepsilon n^2  $.
We choose $v_{i_4},v_{i_5}$ such that $ (v_{i_4},v_{i_5},v_{i_3}) \in E(F'')$, which together with $ (v_{i_1},v_{i_2},v_{i_3}) $ forms a copy of $ \Lambda$.
\end{proof}

\subsection{The proof of  Lemma~\ref{lm2}.}
As mentioned earlier, the core of the proof of  Lemma~\ref{lm2} is to bound the number of small configurations that can locally improve the $\{Y, E\}$-tiling.
These configurations are encoded in the following families $\cT_1$ and $\cT_2$. 
Recall that by the definition of $ G_P $, for $ P\in \binom{\mathcal{Y}}{2} $, if $ P\in \mathcal{D} $, then $ E(G_{P})=\emptyset $.
We consider two families of triples.
Let 
\[\cT_1:=\left\{T\in \binom{\cY}{3}: \text{there is a $\{Y, E\}$-tiling in $ H'[V(T)\cup W] $ covering more than 12 vertices}\right\} \] and 
\[\cT_2:=\left\{T\in \binom{\cY}{3}\setminus \cT_1: f(T)\ge17, g(T)>37\right\}. \]
We will bound the sizes of these two families in  Claim~\ref{c5} and then show the claimed bound for every triple not belonging to either family.

For an arbitrary $ T \in \binom{\cY}{3}\setminus  \cT_1 $, there is no $\{Y, E\}$-tiling on $ V(T)\cup W $ covering more than 12 vertices. 	
A \emph{star} is a graph all of whose edges share a common vertex.
We call the common vertex the \emph{center} of the star. 
For convenience, by a \emph{transversal vertex cover}, we mean a vertex cover with one vertex from each part.
\begin{claim}\label{vtxcov}
For $ T \in \binom{\cY}{3}\setminus  \cT_1 $, if $ f(T)\ge17 $, then	$G_{T}$ has a transversal vertex cover of size three.
\end{claim}
\begin{proof}
Let $ T=:\{Y_{i_1},Y_{i_2},Y_{i_3}\}$, $ V_{j}:=V(Y_{i_j})$ for $j\in[3] $, $ G_1:=G_{\{Y_{i_1},Y_{i_2}\}} $, $ G_2:=G_{\{Y_{i_2},Y_{i_3}\}} $ and $ G_3:=G_{\{Y_{i_3},Y_{i_1}\}} $. Recall that $ \mathcal{D} $ is the collection of $ P\in \binom{\cY}{2} $ such that $ e(G'_P)\ge6 $  and $ G'_P $ has a vertex cover of two vertices which are in the same part.
We may assume that $ \{Y_{i_j},Y_{i_k}\}\notin \mathcal{D}$ for $ j,k\in[3] $, as otherwise $ f(T)\le2\times8=16 $, which is a contradiction.
Note that $ G_{T} $ contains no matching of size four, as otherwise, there exists a $ Y $-tiling of size four containing this matching and eight vertices in $ W $, contradicting that $ T\notin \cT_{1} $.
Without loss of generality, suppose $ e(G_{1})\ge e(G_{2})\ge e(G_{3}) $.
Since $ \{Y_{i_j},Y_{i_k}\}\notin \mathcal{D}$ for $ j,k\in[3] $ and by Claim~\ref{c3}, $ G_{i} $ contains no matching of size three, we have $ e(G_{i})\le7$ for $i\in[3]  $.
Since $ f(T)\ge17 $, we have $ e(G_{1})\ge6 $ and $ e(G_{3})\ge3  $.
We suppose that $ v_1\in V_1, v_2\in V_2 $ form a  vertex cover of $ G_1 $.

We claim that $ G_{T}-\{v_1,v_2\} $ has no matching of size two.
Indeed, suppose $e_1,e_2\in E(G_{T}-\{v_1,v_2\}) $ are disjoint.
Note that  $ E(G_{1}-\{v_1,v_2\})=\emptyset $, therefore $ |(e_1\cup e_2)\cap V_1|+|(e_1\cup e_2)\cap V_2|=2  $.
For the case $ e(G_{1})=7 $, we have $  d_{G_1}(v_1)=d_{G_1}(v_2)=4 $ and  we can greedily find $e_3\in E(G_{T}) $  containing $ v_1 $ and $ e_4\in E(G_{T}) $ containing $ v_2 $, such that $ \{e_1,e_2,e_3,e_4\} $ is a matching of size four, which is a contradiction.
For the case $ e(G_{1})=6 $, we see that $ e(G_{2})=6 $, $ e(G_{3})\ge5 $.
So $ e(G_{i}-\{v_1,v_2\})\ge1 $ for $ i\in\{2,3\} $.
We may assume $ e_1\in E(G_2)$ and $e_2\in E(G_3) $, as if $ e_1,e_2\in E(G_{i})  $, where $ i\in \{2,3\} $, then we can find $ e' $ in the other bipartite graph $ G_{j} $, $ j\in \{2,3\} $, $ j\not=i $ such that  $ e' $ is disjoint from $ e_1 $ or $ e_2 $.
Note that $ |(e_1\cup e_2)\cap V_1|=1 $ and $ |(e_1\cup e_2)\cap V_2|=1  $.
Since $ e(G_{1})=6 $, we have $ d_{G_1-v_2}(v_1)$, $d_{G_1-v_1}(v_2)\ge2 $.
So we can also greedily find $e_3,e_4\in E(G_{T}) $  containing $ v_1,v_2 $, such that $ \{e_1,e_2,e_3,e_4\} $ is a matching of size four, which is a contradiction.
Thus $ G_{T}-\{v_1,v_2\} $ has no matching of size two, implying $ G_{T}-\{v_1,v_2\} $ is a star or a triangle.
Since  $ E(G_{1}-\{v_1,v_2\})=\emptyset $, we get that $ G_{T}-\{v_1,v_2\} $ is a star with the center denoted by $ v_3 $.
If $ v_3\in V_1$ (or $V_2$), then the vertex cover of $ G_2 $(or $G_3$) is a single vertex  and the vertex cover of  $ G_3 $ (or $G_2$) is either a single vertex or a pair of vertices in the same part. 
Note that there are at most  four edges in $ G_2 $ (or $G_3$) and at most five edges in $ G_3 $ (or $G_2$) because $ \binom{T}{2} \cap \mathcal{D}=\emptyset $. 
Thus, we obtain $ f(T)\le 7+4+5=16 $, which is a contradiction.
So $ v_3\in V_3 $ and $\{ v_{1},v_{2},v_{3}\}$  is  a transversal vertex cover of $G_{T}$.
\end{proof}

We have the following fact about $ G_{T} $.
\begin{fact} \label{f2}
Suppose $ T:=\{Y_{i_1},Y_{i_2},Y_{i_3}\}\in \binom{\cY}{3}\setminus  \cT_1 $ such that $ Q_T $ has a perfect matching $\{e_1,e_2,e_3,e_4\} $. 
Then the following holds for all distinct $ p, q\in [4] $.
\begin{enumerate}
	\item The number of edges in $ G_T[e_p,e_q]$ is no more than two and  $ G_T= \bigcup_{ p,q\in[4], p\neq q}G_T[e_p,e_q]$.  In particular, $ f(T)\le12 $.  
	\item   If $ e(G_T[e_p,e_q])=1 $, then $ e(Q_T[e_p,e_q])\le6 $;
	if	$ e(G_T[e_p,e_q])=2 $, then $ e(Q_T[e_p,e_q])\le5 $.
\end{enumerate}
\end{fact}
\begin{proof}
Let $ V_1:=V(Y_{i_1})=\{h_1,h_2,h_3,h_4\} $, $ V_{2}:=V(Y_{i_2})=\{k_{1},k_{2},k_{3},k_4\}$, and $V_{3}:=V(Y_{i_3})=\{g_{1},g_{2},g_{3},g_4\}$ such that $ e_i=\{h_i,k_i,g_i\}$ for $ i\in[3] $.
We claim that $ E(G_T[e_p])=\emptyset$ for $ p\in[4] $. 
Indeed, without loss of generality, suppose $h_1k_1\in E(G_T) $. 
By the definition of $G_{T}$, we can find vertices $w_1,w_2\in W$ such that $h_1,k_1,w_1,w_2$ form a copy of $Y$, which together with $e_2,e_3,e_4 $ forms a $\{Y, E\}$-tiling covering 13 vertices, contradicting that $ T\notin \cT_1 $.
Thus  $ G_T=\bigcup_{ p,q\in[4], p\neq q}G_T[e_p,e_q]$.

Next we claim that $ G_T[e_1,e_2] $ has no matching of size two.  
Suppose $ G_T[e_1,e_2] $ has two disjoint edges $ e_1' $ and $ e_2' $. 
Thus we can find two disjoint copies of $ Y $ containing these two edges and four vertices in $ W $. 
These together with $ e_3 $ and $ e_4 $ give  a $\{Y, E\}$-tiling covering 14 vertices, contradicting that $ T\notin \cT_1 $. 
So  $ G_T[e_1,e_2] $ has no matching of size two.
By K\"{o}nig's theorem, there is a vertex cover in $ G_T[e_1,e_2] $ of size one.
Note that $ E(G_T[V_i])=\emptyset$ for $ i\in[3] $.
Thus $ e(G_T[e_1,e_2])\le2$ and by symmetry we get $ f(T)=\sum_{ p,q\in[4], p\neq q}e(G_T[e_p,e_q])\le 6\times 2=12  $.

For (2) of the fact, for the case $ e(G_T[e_1,e_2])=1 $, we suppose $h_2k_1\in E(G_T) $.
Thus we have $h_1k_2g_1\notin E(Q_T) $, because otherwise there exist $w_1,w_2\in W$ such that $h_2k_1w_1, h_2k_1w_2\in E(H)$, together with $ e_3,e_4 $ and $h_1k_2g_1$  give  a $\{Y, E\}$-tiling covering 13 vertices, a contradiction. Similary, we have $ h_1k_2g_2\notin E(Q_T) $.
Thus $ e(Q_T[e_1,e_2])\le 8-2=6 $.

For the case $ e(G_T[e_1,e_2])=2 $, suppose  $ e_1',e_2'\in E(G_T[e_1,e_2]) $. 
Since $ G_T[e_1,e_2] $ has no matching of size two, we have $ |e_1'\cap e_2'|=1 $.
Without loss of generality, suppose $h_2k_1,g_2k_1\in E(G_T)  $, similar arguments as in the last paragraph, we get $ h_1k_2g_1,h_1k_2g_2,h_2k_2g_1\notin E(Q_T) $. 
Thus $ e(Q_T[e_1,e_2])\le 8-3=5 $.
\end{proof}

Note that any such $ \{Y, E\} $-tiling in the definition of $ \cT_1 $ uses an edge of the form $ YYW $, and by the definition of $ H' $, for such a pair in $ T $, there are at least $ 16\eps n $ choices for the vertex in $ W $. 
This flexibility is crucial in the following claim, which shows that both $ \cT_1 $ and $ \cT_2 $ are small.
\begin{claim}\label{c5}
$ |\cT_1|+|\cT_2|\le \varepsilon'n^{3} $.
\end{claim}
\begin{proof}
We claim that $ \cT_1$ does not contain $ \varepsilon n $ triples of members of $ \cY $ that are  pairwise disjoint.
Suppose otherwise, let $ \cT_1'$ be a subset of $ \cT_1$ consisting of  $ \varepsilon n $ pairwise disjoint triples in $\binom{ \cY}{3} $.
Note that two triples $ T_1 $ and $T_2$ in $ \binom{ \cY}{3} $ are disjoint if and only if $ T_1\cap T_2=\emptyset $.
By the definition of $ \cT_1$, for any $T\in \cT_1'$ we can find  a $\{Y, E\}$-tiling in $ H'[V(T)\cup W]$ covering at least $ 13 $ vertices but no more than 16 vertices (by excluding members of the tiling if necessary).
Note that the $ \{Y, E\} $-tiling consists of edges of the form $ YYY $ and $ YYW $ and thus at most half of the vertices in each member of the $\{Y, E\}$-tiling are from $ W $.
So the $ \{Y, E\} $-tiling covers at most 8 vertices in $ W $. 
By the definition of $ G_T $, for each $\{u, v\}\in E(G_T)$, there are at least $ 16\varepsilon n $ vertices $w$ in $W$ such that $ \{u,v,w\}\in E(H) $.
Thus for each edge of the form $ YYW $  in the above $ \{Y, E\} $-tiling, there are at least $ 16\eps n $ choices for the vertex in $ W $.
Since  $ |\cT_1'|=\eps n $,
we can find $ W'\subseteq W $ with  $ |W'|\le8\varepsilon n$ such that $ H'\left[W'\cup (\bigcup_{T\in \cT_1'}V(T))\right] $ contains a $\{Y, E\}$-tiling covering at least $ 13\varepsilon n $ vertices. 
Together with the members of $ \cY $ that are not in any triple of $ \cT_1'$, we get a $\{Y, E\}$-tiling covering at least $ 4t_2+\varepsilon n $ vertices. 
By Proposition~\ref{prop}, there exists a $Y$-tiling $\mathcal{Y}'  $ in $ H[R\cup(W\setminus W')] $ of size $t_1$. 
Thus we have a $\{Y, E\}$-tiling covering $ 4t_1+4t_2+\varepsilon n $ vertices,  contradicting our assumption that there exists no $\{Y,E\}$-tiling in $ H $ covering $ 4|\cT|+\varepsilon n $ vertices. 

Suppose $ \cT_1''$ is a maximal subset of $ \cT_1 $ such that members of $ \cT_1'' $ are pairwise disjoint. 
Then other triples in $ \cT_1 $ must have nonempty intersection with some member  in $\cT_1''$ and $ |\cT_1''|\le\eps n-1$. 
Thus 	
\[|\label{1.7} \cT_1|\le 3(\varepsilon n-1)t_{2}^{2}\le\varepsilon'n^{3}/2. \tag{4.7}\]

Now we prove $|\cT_2|\le \varepsilon'n^{3}/2 $. 
For any triple $ T=\{Y_{i},Y_{j},Y_{k}\}\in \cT_2$, we have $ f(T)\ge17,  g(T)\ge38$.	
Towards a contradiction, we may assume 	$| \cT_2|> \varepsilon'n^{3}/2 $. 
Define an oriented $3$-graph $ F $ on $ \cY $. 
The ordered triple $ (Y_i,Y_j,Y_k )$ is an oriented edge if  $ \{Y_i,Y_j,Y_k\}\in \cT_2 $ and 
$ e(G[V(Y_i),V(Y_j)])\ge6 $.
Thus  $e(F)> \varepsilon'n^{3}/2>12\varepsilon n^{3} $. 
Applying Lemma~\ref{f4}, 
we get  a family of $\varepsilon n$  vertex-disjoint copies of $ \Lambda$, which is denoted by  $ \mathcal{P} $. 

Consider a member of $ \mathcal{P} $ with edges   $ (Y_{i_{1}},Y_{i_{2}},Y_{i_{3}} )$ and $ (Y_{i_{4}},Y_{i_{5}},Y_{i_{3}} ) $. 
Let  $ T':=\{Y_{i_{1}},Y_{i_{2}},Y_{i_{3}}\}$, $T'':=\{Y_{i_{4}},Y_{i_{5}},Y_{i_{3}}\}$, and $V_{k}:= V(Y_{i_{k}})$ for $ k\in[5] $.
Then we have the corresponding $3$-partite graphs $ G_{ T'},G_{ T''} $ and $3$-partite $3$-graphs $ Q_{ T'},Q_{ T''} $.
Let $ G:= G_{T'}\cup G_{T''} $.
Note that $ f(T')\ge17 $ and $ f(T'')\ge17  $.
Applying Claim~\ref{vtxcov}, we suppose that $ \{v_1,v_2,v_3\} $ is the transversal vertex cover of $ G_{ T'} $ and $ \{v_3',v_4,v_5\} $ is the transversal vertex cover of $ G_{ T''} $ with  $ v_3'\in V_3 $, $ v_i\in V_i $ for $ i\in[5] $.
Since $ g(T')\ge38 $ and the number of edges intersecting $ \{v_1,v_2,v_3\} $ is at most 37, there exists $ e\in Q_{T'} $ such that $ e\cap\{v_1,v_2,v_3\}=\emptyset $.
Note that $\{v_i, v_{i+1}\} $ is a vertex cover of $ G[V_i, V_{i+1}] $ for $i\in \{1,4\}$.
Then  $ 6\le e(G[V_i, V_{i+1}])\le7 $ for $i\in \{1,4\}$ and $ d_{G_{T''-\{v_4,v_5\}}}(v_3')\ge17-(8\times2-1)=2 $.
So we can find disjoint $ e_i\in E(G[V_1, V_{2}]-e) $ containing $ v_i $ for $ i\in[2] $ and $ e_3\in E(G_{ T''-\{v_4,v_5\}}) $ containing $ v_3' $.
Finally, we find disjoint $ e_j\in E(G[V_4, V_{5}]-e_3) $ containing $ v_j $ for $ j\in\{4,5\} $.
Thus $ \{e_i:i\in[5]\} $ is a matching of size five in $ G-e $, see Figure~\ref{figure2}.
Since each edge of $ \{e_i:i\in [5]\} $ can be extended to a copy of $Y$ by using two vertices in $W$, we may replace $T'\cup T''$ by five disjoint copies of $Y$ and $ e $ to obtain a $\{Y, E\}$-tiling  covering $ 5\times4+3=23>20 $ vertices. 

Applying this argument in the previous paragraph  to each member of $ \mathcal{P} $, we get that there exists $ W' \subseteq W $ with $ |W'|=10\varepsilon n $ such that $ V(\mathcal{P})\cup W'  $ contains a $\{Y, E\}$-tiling covering $ 23\varepsilon n $ vertices. 
Together with $ \cY $, we get a $\{Y, E\}$-tiling $ \cT' $ covering  $ 4t_2+3\varepsilon n $ vertices.
Due to Proposition~\ref{prop}, there exists a $Y$-tiling $\mathcal{Y}'  $ in $ H[R\cup(W\setminus W')] $ of size $t_1$. 
Thus $\mathcal{Y}'\cup \cT' $ is a $\{Y, E\}$-tiling  covering  $ 4t_1+4t_2+3\varepsilon n $ vertices,  contradicting our assumption on $ H $. 
Hence 	$| \cT_2|\le  \varepsilon'n^{3}/2 $. 
Combining with  \eqref{1.7}, we get	$ |\cT_1 \cup \cT_2 |\le \varepsilon'n^{3} $.	
\end{proof}

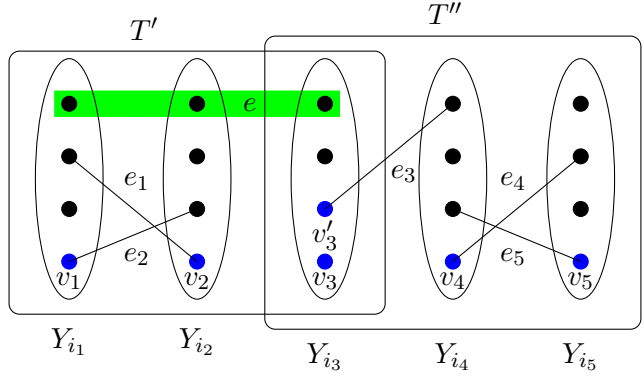
\begin{figure}[h]
	\begin{tikzpicture}
		[inner sep=2pt,
		vertex/.style={circle, draw=black!50, fill=black!50},
		]
		\draw[line width=10pt, color=green] (-1.2,1)--(2.6,1);
		\draw[rounded corners] (-1.8,1.7) rectangle (3.2,-1.8);
		\draw[rounded corners] (1.6,1.9) rectangle (6.6,-2);
		\draw (-1,0) ellipse (0.45 and 1.6);
		\draw (0.7,0) ellipse  (0.45 and 1.6);
		\draw (2.4,0) ellipse  (0.45 and 1.6);
		\draw (4.1,0) ellipse  (0.45 and 1.6);
		\draw (5.8,0) ellipse  (0.45 and 1.6);
		
		\node at (-1,1) [vertex,color=black] {};
		\node at (-1,0.3) [vertex,color=black] {};
		\node at (-1,-0.4) [vertex,color=black] {};
		\node at (-1,-1.1) [vertex,color=blue] {};
		
		\node at (0.7,1) [vertex,color=black] {};
		\node at (0.7,0.3) [vertex,color=black] {};
		\node at (0.7,-0.4) [vertex,color=black] {};
		\node at (0.7,-1.1) [vertex,color=blue] {};
		\node at (2.4,1) [vertex,color=black] {};
		\node at (2.4,0.3) [vertex,color=black] {};
		\node at (2.4,-0.4) [vertex,color=blue] {};
		\node at (2.4,-1.1) [vertex,color=blue] {};
		\node at (4.1,1) [vertex,color=black] {};
		\node at (4.1,0.3) [vertex,color=black] {};
		\node at (4.1,-0.4) [vertex,color=black] {};
		\node at (4.1,-1.1) [vertex,color=blue] {};
		\node at (5.8,1) [vertex,color=black] {};
		\node at (5.8,0.3) [vertex,color=black] {};
		\node at (5.8,-0.4) [vertex,color=black] {};
		\node at (5.8,-1.1) [vertex,color=blue] {};
		
		\draw (-1,0.3) -- (0.7,-1.1);
		\draw (0.7,-0.4) -- (-1,-1.1);
		\draw (5.8,0.3) -- (4.1,-1.1);
		\draw (4.1,-0.4) -- (5.8,-1.1);
		\draw (2.4,-0.4) -- (4.1,1);

		\node at (1.4,0.75) [label=above:$e$] {};
		\node at (-0.1,-0.3) [label=above:$e_1$] {};
		\node at (-0.1,-1.3) [label=above:$e_2$] {};
		\node at (4.9,-0.3) [label=above:$e_4$] {};
		\node at (3.45,-0.2) [label=above:$e_3$] {};
		\node at (4.9,-1.3) [label=above:$e_5$] {};
		
		\node at (-1,-1.62) [label=above:$v_{1}$] {};
		\node at (0.7,-1.62) [label=above:$v_{2}$] {};
		\node at (2.4,-1.62) [label=above:$v_{3}$] {};
		\node at (2.4,-1.1) [label=above:$v_{3}'$] {};
		\node at (4.1,-1.62) [label=above:$v_{4}$] {};
		\node at (5.8,-1.62) [label=above:$v_{5}$] {};
		\node at (-1,-2.5) [label=above:$Y_{i_{1}}$] {};
		\node at (0.7,-2.5) [label=above:$Y_{i_{2}}$] {};
		\node at (2.4,-2.7) [label=above:$Y_{i_{3}}$] {};
		\node at (4.1,-2.7) [label=above:$Y_{i_{4}}$] {};
		\node at (5.8,-2.7) [label=above:$Y_{i_{5}}$] {};
		\node at (0,1.7) [label=above:$	T'$] {};
		\node at (4,1.9) [label=above:$	T''$] {};
		
	\end{tikzpicture}
	\caption{Five disjoint edges avoiding $ e$ in $ G_T $ in a copy of $ \Lambda$.}
	\label{figure2}
\end{figure}

Finally we are ready to prove Lemma~\ref{lm2}.
\begin{proof}[Proof of Lemma~\ref{lm2}]
Fix an arbitrary $ T \in \binom{\cY}{3}\setminus  (\cT_{1}\cup\cT_{2}) $.
Suppose $ T=\{Y_{i_1},Y_{i_2},Y_{i_3}\}$.
Let $ V_1:=V(Y_{i_1})=\{h_1,h_2,h_3,h_4\} $, $ V_{2}:=V(Y_{i_2})=\{k_{1},k_{2},k_{3},k_4\}$ and $V_{3}:=V(Y_{i_3})=\{g_{1},g_{2},g_{3},g_4\}$.
By Claim~\ref{c3}, we have $ e(G_{P})\le8$ for $P\in\binom{\cY}{2}  $.
So $ f(T)\le8\times 3=24 $.
We shall show $ g(T)\le 64-\dfrac{9}{8}f(T) $ in several cases depending on the value of $ f(T) $. More precisely, we prove the following bounds in Table~\ref{table}. 
\begin{table}[!hbp]
\begin{tabular}{|c|c|c|c|c|c|c|}
\hline

\hline

$ f(T) $ & 0 &  $ 1\sim 10  $ &  $ 11\sim 14 $ &  $ 15\sim 16 $ &   $ 17\sim 24 $\\

\hline

$ g(T) $ & $ \le64 $ & $ \le52 $ & $ \le48 $ & $ \le46 $  & $ \le37$ \\

\hline

\end{tabular}
\caption{\label{table} A table summarizing the proof of Lemma~\ref{lm2}.}
\end{table} 

\bm{$ f(T)=0 $}. 
The trivial upper bound on $ g(T)$ is 64.

\bm{$1\le f(T)\le10 $}. 
Suppose $ h_1k_1\in E(G_{T}) $. 
We claim that  $Q'_{T}:=Q_{T}-\{h_1,k_1\}$ has no matching of size three. 
Otherwise $ \{h_1,k_1\}\cup W $ contains a copy of $ Y $, which together with a matching of size three in  $Q'_{T}$ forms a  $\{Y, E\}$-tiling on $ V(T)\cup W $ covering  13 vertices, contradicting that $ T\notin \cT_{1} $.
So applying Fact~\ref{f0} to $ Q'_{T} $, we have $e(Q'_{T})\le24.$ 
Let $ E_{h_1,k_1} :=\{xyz\in E(Q_{T}):\{x,y,z\}\cap\{h_1,k_1\}\not= \emptyset\}$.
Note that $ |E_{h_1,k_1}|\le4^{3}-3\times3\times4=28. $ Hence
\[
g(T)= e(Q'_{T})+|E_{h_1,k_1}|\le 24+28=52.
\]

\bm{ $11\le f(T)\le14 $}. 
If $ Q_{T} $ has no perfect matching, then by Fact~\ref{f0}, $  g(T)\le 48$. 	
Otherwise, suppose $ Q_T $ has a perfect matching $\{e_1,e_2,e_3,e_4 \}$, then by Fact~\ref{f2}, $f(T)\le12  $.    
Let $ a_1:=|\{(p,q):  e(G_T[e_p,e_q])=1\}| $ and $ a_2:=|\{(p,q):  e(G_T[e_p,e_q])=2\}| $. 
Applying Fact~\ref{f2}, we have $ f(T)=\sum_{ p,q\in[4], p\neq q}e(G_T[e_p,e_q])=a_1+2a_2$ with $ a_1+a_2\le6 $ and  $  g(T)\le 4^{3}-(8-6)a_1-(8-5)a_2 $.
Since $11\le f(T)=a_1+2a_2\le 14 $ and $ a_1+a_2\le6 $, we have  $11\le f(T)\le12 $ and $ a_1+a_2=6 $.
If $ f(T)=11 $, then $ a_1=1, a_2=5 $, so $  g(T)\le 4^{3}-(8-6)-5(8-5) =47$.
If $ f(T)=12 $, then $ a_1=0, a_2=6 $, so $  g(T)\le 4^{3}-6(8-5) =46$.
Thus $  g(T)\le 48$.

\bm{$15\le f(T)\le16 $}. 
By $ f(T)>12  $ and Fact~\ref{f2}, we get that $ Q_{T} $ has no perfect matching.
Then by Fact~\ref{f0}, $  g(T)\le 48$. 	 
Suppose   $g(T)\in\{47,48\}$. 
Since $ Q_{T} $ has no matching of size four and $ g(T)\in\{47,48\} $, using Fact~\ref{f1} and Fact~\ref{f11},  we get $  Q_{T}\cong  H^{(3)}(3,4)^{-} $ or $  Q_{T}\cong  H^{(3)}(3,4) $. 
Without loss of generality, suppose $ I\subseteq V_2 $ is a set of three vertices which are in the same part as the isolated vertex of $ Q_{T} $. 
Note that $ I $ is a vertex cover of $ Q_{T} $.   
We observe that $ I $ is also a vertex cover of $ G_{T} $.  
Indeed, suppose $e\in E(G_{T}-I) $, then we can find a matching of size three in $Q_{T}-e  $ because  $ H^{(3)}(3,4)^{-}\subseteq Q_{T} $.  
Since $e$ can be extended to a copy of $Y$ by using two vertices in $W$, we get a $ \{Y, E\} $-tiling on $ V(T)\cup W $ with 13 vertices, contradicting that $ T\notin \cT_1 $.
So $ I $ is a vertex cover of $ G_{T} $ and $ E(G_{T}[V_1\cup V_3])=\emptyset $.  
By the definitions of $ G_{T} $ and $ \mathcal{D} $, we get that for $P \subseteq T$, if $P \in \mathcal{D}$, then $ E(G_{P})=\emptyset $, otherwise $ e(G'_P)\le5 $.
Thus we have $ f(T)\le 5\times2=10 $, which is a contradiction.
So  $ g(T)\le46$. 

\bm{$ 17\le f(T)\le 24 $}. 
Since $ T\notin \cT_{2} $, we have $ g(T)\le37$.

In summary, for any $ T \in \binom{\cY}{3}\setminus  (\cT_{1}\cup\cT_{2}) $, we have $ g(T)\le 64-\dfrac{9}{8}f(T) $. 
By Claim~\ref{c5}, we get that $|\cT_{1}\cup\cT_{2}|\le\varepsilon' n^{3}$, completing the proof of Lemma~\ref{lm2}.
\end{proof}
\subsection{The proof of Fact~\ref{f1}}
Suppose $ H:=H[V_1,\dots,V_k] $ is a $ k $-partite $ k $-graph with $ |V_i|=n$ for $ i\in[k] $ and $ e(H)=tn^{k-1}-1 $, and $ H $ contains no matching of size $ t+1 $. 
Our goal is to show $ H\cong  H^{(k)}(t,n)^{-} $.
For $ k=2 $, since $ tn-1> (t-1)n $, by Fact~\ref{f11}, $ H $ has a matching of size $ t $. By K\"{o}nig's theorem \cite{2000Graph}, there is a vertex cover in $ H $ of size $ t $.
As $ t\ge 3 $, to cover all $ tn-1 $ edges, these $ t $ vertices must be in the same part. 
Thus $ H\cong H^{(2)}(t,n)^{-} $.

Now we assume $ k>2 $. 
Consider the complete $ (k-1) $-partite $ (k-1) $-graph on $ (V_1,\dots, 
V_{k-1}) $ denoted by $ K^{(k-1)} $ and a matching $ M $ of it.
We define an auxiliary bipartite graph $ F_M $ with vertex classes $ M $ and $ V_k $ such that $ \{e,v\}\in E(F_M) $, $ e\in M $, $ v\in V_k $ if and only if $ e\cup\{v\}\in E (H)$.
Let $ M_1,\dots, M_{n^{k-2}} $ be a 1-factorization of $ K^{(k-1)} $, that is, a decomposition of $ E( K^{(k-1)}) $ into $ n^{k-2} $ perfect matchings.
Since $ H $ has no matching of size $ t+1 $, so does $ F_{M_i} $, $ i\in[n^{k-2}] $, as otherwise a matching of size $ t+1 $ in $ F_{M_i} $ gives a matching of size $ t+1 $ in $ H $, which is a contradiction.
So using Fact~\ref{f11}, we have $ e(F_{M_i})\le tn $. 
As $ \sum_{i=1}^{n^{k-2}}e(F_{M_i}) = e(H)=tn^{k-1}-1  $, there exists $ p\in[n^{k-2}] $ such that $ e(F_{M_p})= tn-1 $ and  $ e(F_{M_i})= tn $ for $ i\in[n^{k-2}]\setminus \{p\} $.
By the case $ k=2 $ and Fact~\ref{f11}, for $ i\in[n^{k-2}] $ we have $ F_{M_i}\cong  H^{(2)}(t,n)^{-}  $ or $ F_{M_i}\cong  H^{(2)}(t,n)  $, that is, there is a vertex cover $ C_i $ of size $ t $ in $ F_{M_i} $ such that either $ C_i \subseteq  M_i  $ (we say such $ M_i $ is of type I) or $ C_i \subseteq V_k $  (we say such $ M_i $ is of type II).

Since every perfect matching $ M $ in $ K^{(k-1)} $ belongs to a 1-factorization, for any matching $ M $ in $ K^{(k-1)} $, there is a vertex cover $ C_{M} $ in $ F_{M} $ of size $ t $ such that $ C_M \subseteq  M  $ (type I) or $ C_M \subseteq V_k $ (type II).
We say that  two perfect matchings $ M' $ and  $ M'' $ of $ K^{(k-1)} $ are equivalent,  if both $ M'$ and $ M'' $ are of type II with $ C_{M'}=C_{M''} $ or  both $ M'$ and $ M'' $ are of type I, and write  $M'\simeq M''$.
Note that for any edge  $ e $ of $ K^{(k-1)} $, the set $ N_{F_{M}}(e) = N_H(e) $ is the same for all $ M $ containing $ e $.
Using this we show the following claim.
\begin{claim}\label{1.10}
Suppose $ M'$ and  $M'' $ are perfect matchings of $ K^{(k-1)} $ with $ |M'\cap	M''|\ge1 $, then  $M'\simeq M''$. 
\end{claim}
\begin{proof}
First of all, we show that if $ |M'\cap M''|\ge 2 $, then $M'\simeq M''$.
Indeed, we suppose that $ M' $ and $ M'' $ are of different types. 
Without loss of generality, suppose $ M' $  is of  type I and $ M'' $ is of  type II, and $ e_1,e_2\in M'\cap M'' $.
For each $ i\in[2] $, we have $ |N_{F_{M''}}(e_i)|\in \{t-1,t\}$ and  $|N_{F_{M'}}(e_i)|\in \{0,n-1,n\}$. Note that at most one of $ |N_{F_{M'}}(e_1)| $ and $ |N_{F_{M'}}(e_2)| $ is $ n-1 $,
 and $ N_{F_{M'}}(e_i)=N_{F_{M''}}(e_i)=N_{H}(e_i) $ for $ i\in[2] $.
So \[2n-1\le\sum_{i\in[2]}|N_{F_{M'}}(e_i)|=\sum_{i\in[2]}|N_{H}(e_i)|=\sum_{i\in[2]}|N_{F_{M''}}(e_i)|\le2t,\] contradicting that $ t\le n-1 $.
Thus, $M'$ and $ M''$ are of the same type.
Suppose both $ M'$ and  $ M'' $ are  of  type II, then the  vertex covers satisfy $ C_{M'}=C_{M''} $.
Indeed, suppose $ C_{M'}\neq C_{M''} $, then there exists $ e_i$, where $i\in [2] $, such that  $ N_{F_{M'}}(e_i)\not=N_{F_{M''}}(e_i)$, which is a contradiction. 
So $M'\simeq M''$.

Secondly, we suppose $ e_1\in M'\cap M'' $, $ e_2\in M' $ and $ e_3\in M'' $, where $ e_1\cap e_2=e_1\cap e_3=\emptyset $.
Since $ n\ge4 $, we can choose $ e_4\in E( K^{(k-1)}) $ disjoint from $ e_i $, $ i\in[3] $.
Consider  perfect matchings $ M_0' $ and $ M_0'' $ of $ K^{(k-1)} $ such that $ e_1,e_2,e_4\in  M_0' $ and $ e_1,e_3,e_4\in  M_0'' $. 
Since $ |M_0'\cap M'|\ge 2 $, $ |M_0''\cap M''|\ge 2 $ and $ |M_0'\cap M_0''|\ge 2 $, we have $M'\simeq M_0'$, $M''\simeq M_0''$ and $M_0'\simeq M_0''$.
Thus $M'\simeq M''$.
\end{proof}

Next we show that all perfect matchings of $ K^{(k-1)} $ are  equivalent.
For any perfect matchings $ M' $ and $ M'' $ of $ K^{(k-1)} $, suppose $ e_1\in M'' $ and $ e_2\in M' $.
Since $ n\ge4 $, we can choose $ e_3\in E( K^{(k-1)}) $ such that $ e_3\cap (e_1\cup e_2)=\emptyset $.
Let $ M_{e_i} $ be a perfect matching in $ K^{(k-1)} $ containing $  e_i,e_3 $ for each $ i\in[2] $.
So by Claim~\ref{1.10} and transitivity, we have $M'\simeq M_{e_1}\simeq M_{e_2}\simeq M''$.

Finally we fix a 1-factorization  $ M_1,\dots, M_{n^{k-2}} $ of $ K^{(k-1)} $.
If all $ M_i $ are of type II, then there exists $ C\subseteq V_k $ which is a vertex cover of size $ t $ in $ F_{M_{i}} $ for all $ i\in[n^{k-2}] $. 
Therefore, by the definition of $ F_{M_{i}} $, $ C $ is a 
vertex cover of $ H $ of size $ t $.
Since  $ e(H)=tn^{k-1}-1 $, we get $ H\cong H^{(k)}(t,n)^{-} $.
Finally, consider the case that $ C_{M_i} \subseteq M_i $ for all $ i\in[n^{k-2}] $.
Set $ H' :=\bigcup_{i\in[n^{k-2}] }C_{M_i}$. 
Then $ H' $ has $ tn^{k-2} $ $ (k-1) $-sets.  
Note that each $ C_{M_i} $ contains $ t $ edges and all but one of these $ t $ edges have degree $ n $ in $ H $, while the last one has degree $ n $ or $ n-1 $ in $ H $ and in total at most one edge has degree $ n-1 $.
We see that $ H' $ has no matching of size $ t + 1 $, as otherwise this matching could be extended to a matching of size $ t + 1 $ in $ H $, which is a contradiction. 
Applying Fact~\ref{f11}, we get $ H'\cong  H^{(k-1)}(t,n)  $, which implies $ H\cong   H^{(k)}(t,n)^{-}  $.

\section*{Acknowledgment}
We would like to thank Donglei Yang for valuable discussions and anonymous referees for their valuable comments that greatly improved the presentation of this paper.
Guanghui Wang was supported by  National Key R\&D Program of China (2020YFA0712400), Natural Science Foundation of China (12231018) and Young Taishan Scholars probgram of Shandong Province (201909001).	Jie Han was supported by  Natural Science Foundation of China (12371341).

	
\bibliographystyle{abbrv}
\bibliography{LargeYfinal}

\end{document}